\documentclass{amsart}%
\usepackage{amsmath}
\usepackage{amssymb}
\usepackage{amscd}
\usepackage{amsfonts}
\usepackage{enumerate}
\usepackage{mathrsfs}
\usepackage{xy}
\usepackage{stmaryrd}
\usepackage{graphicx}
\usepackage{color}
\usepackage{multirow}
\usepackage{exscale}
\usepackage{enumitem,array}%
\usepackage{subfigure}
\usepackage{extarrows}
\usepackage{amsfonts}

\newtheorem{mytheo}{Theorem}[section]
\newtheorem{myexp}{Problem}[section]

\newtheorem{lem}[mytheo]{Lemma}
\newtheorem{preli}[mytheo]{Preliminaries}

\newtheorem{prop}[mytheo]{Proposition}

\newcounter{remark}

\newcounter{problem}

\makeatletter
\def\@upcite#1#2{\textsuperscript{[{#1\if@tempswa , #2\fi}]}}
\makeatother

\title{Functionally-fitted energy-preserving methods for solving  oscillatory nonlinear Hamiltonian systems}

\def\shortTitle{Energy-Preserving Methods}

\def\myAMS{65L05, 65L06, 65L60, 65P10}

\def\myAbstract{
In the last few decades, numerical simulation for nonlinear oscillators has received
a great deal of attention, and many researchers have been concerned with the design and analysis
of numerical methods for solving oscillatory problems. In this paper, from the perspective of the
continuous finite element method, we propose and analyze new energy-preserving functionally fitted
methods, in particular trigonometrically fitted methods of an arbitrarily high order for solving oscillatory nonlinear Hamiltonian systems with a fixed frequency. To implement these new methods in a
widespread way, they are transformed into a class of continuous-stage Runge--Kutta methods. This
paper is accompanied by numerical experiments on oscillatory Hamiltonian systems such as the FPU
problem and nonlinear Schr\"odinger equation. The numerical results demonstrate the remarkable
accuracy and efficiency of our new methods compared with the existing high-order energy-preserving
methods in the literature.}

\begin{document}

\bibliographystyle{amsplain}
\author[]{Yu-Wen Li}
\address{Department of Mathematics, Nanjing University, Nanjing 210093,
P.R.China}
\email{farseer1118@sina.cn}
    
\author[]{Xinyuan Wu}
\address{Department of Mathematics, Nanjing University; State Key Laboratory
for Novel Software Technology at Nanjing University, Nanjing 210093,
P.R.China}
\email{xywu@nju.edu.cn}
\subjclass[2010]{Primary \myAMS}
\date{July 06, 2016}
\begin{abstract}\myAbstract\end{abstract}
\maketitle
\markboth{YU-WEN Li AND XINYUAN WU}{\shortTitle}


\section{Introduction}
In this paper, we consider nonlinear Hamiltonian systems:
\begin{equation}\label{IVP}
y^{\prime}(t)=f(y(t))=J^{-1}\nabla H(y(t)),\quad y(t_{0})=y_{0}\in\mathbb{R}^{d},
\end{equation}
where $d=2d_{1}, f: \mathbb{R}^{d}\rightarrow\mathbb{R}^{d}, H
: \mathbb{R}^{d}\rightarrow\mathbb{R}$ are sufficiently smooth
functions and
\begin{equation*}
J=\left(\begin{array}{cc}O_{d_{1}\times d_{1}}&I_{d_{1}\times d_{1}}\\-I_{d_{1}\times d_{1}}&O_{d_{1}\times d_{1}}\end{array}\right)
\end{equation*}
is the canonical symplectic matrix. It is well known that the
flow of \eqref{IVP} preserves the symplectic form $dy\wedge Jdy$ and
the Hamiltonian or
 energy $H(y(t))$. In the spirit of geometric
numerical integration, it is a natural idea to design schemes that
preserve both the symplecticity of the flow and Hamiltonian
function. Unfortunately, a numerical scheme cannot achieve this goal
unless it generates the exact solution (see, e.g. \cite{Hairer2006},
page 379). Hence researchers face a choice between preserving
symplecticity or energy and many of them have given more weight on
the former in the last decades, and readers are referred to
\cite{Hairer2006} and references therein. Whereas investigations on
energy-preserving (EP) methods are relatively insufficient (see,
e.g.
\cite{Betsch2000,Brugnano2010,Brugnano2012b,Celledoni2009,Celledoni2010,French1990,Gonzalez1996,Hairer2010,Mclachlan1999,Tang2012}).
{Comparing to symplectic methods, EP methods are beneficial for
improving nonlinear stability, easier to adapt the time step and
more suitable for the integration of chaotic systems (see, e.g.
\cite{Celledoni2012,Hairer1997,Simos1993,Skeel1993}).}

On the other hand, in scientific computing and modelling, the design
and analysis of methods for periodic or oscillatory systems have
been considered by many authors (see, e.g.
\cite{Bettis1970,Gautschi1961,Hairer2000,Petzold1997,Wang2015,Yang2009}).
Generally, these methods utilize a priori information of special
problems and they are more efficient than general-purpose methods. A
popular approach to constructing methods suitable for oscillatory
problems is {using} the functionally-fitted (FF) condition, namely,
{deriving a suitable method} by requiring it to integrate members of
a given finite-dimensional function space $X$ exactly. If $X$
incorporates trigonometrical or exponential functions, the
corresponding methods are also named by trigonometrically-fitted
(TF) or exponentially-fitted (EF) methods (see, e.g.
\cite{Coleman1996,Ixaru2004,Ozawa2001,Simos1998}).

Therefore, combining the ideas of the EF/TF and structure-preserving
methods is a promising approach to developing numerical methods
which allow long-term computation of solutions to oscillatory
Hamiltonian systems \eqref{IVP}. Just as the research of symplectic
and EP methods, EF/TF symplectic methods have been studied
extensively by many authors (see, e.g.
\cite{Calvo2008,Calvo2010a,Calvo2010b,Franco2007,VandenB2003,VandeVyver2006,Wu2012}).
By contrast, as far as we know, only a few papers paid attention to
the EF/TF EP methods (see, e.g.
\cite{Miyatake2014,Miyatake2015,Wang2012}). Usually the existing
EF/TF EP methods are derived in the context of continuous-stage
Runge--Kutta (RK) methods. The coefficients in these methods are
determined by a system of equations resulting from EF/TF, EP and
symmetry conditions. As mentioned at the end of \cite{Miyatake2014},
it is not easy to find such a system having a unique solution in the
case of deriving high-order methods. What's more, how to verify the
algebraic order of such methods falls into a question. A common way
is to check order conditions related to rooted trees. Again, this is
inconvenient in the high-order setting since the number of trees
increases extremely fast as the order grows. In this paper, we will
construct FF EP methods based on the continuous finite element
method, which is inherently energy-preserving (see, e.g.
\cite{Betsch2000,French1990,Tang2012}). Intuitively, we are expected
to increase the order of the method through enlarging the finite
element space. By adding trigonometrical functions to the space, the
corresponding method is naturally trigonometrically fitted. Thus we
are hopeful of constructing FF EP methods, in particular TF EP
methods, of arbitrarily high order.

The outline of this paper is as follows. In Section \ref{TFCFE}, we
construct FF CFE methods and present important geometric properties
of them. In Section \ref{CRKK}, we interpret them as
continuous-stage Runge--Kutta methods and analyse the algebraic
order. We then discuss implementation details of these new methods
in Section \ref{IMPLE}. Numerical results  are shown in section
\ref{NE}, including the comparison between our new TF EP methods and
other prominent structure-preserving methods in the literature. The
last section is concerned with the conclusion and discussion.

\section{Functionally-fitted continuous finite element methods for Hamiltonian systems}\label{TFCFE}
\begin{preli}
Throughout this paper, we consider the IVP \eqref{IVP}
on the time interval $I=[t_{0},T]$, which is equally partitioned
into $t_{0}<t_{1}<\ldots<t_{N}=T$, with $t_{n}=t_{0}+nh$ for
$n=0,1,\ldots,N$. A function space
$Y$=span$\left\{\varphi_{0},\ldots,\varphi_{r-1}\right\}$ means that
\begin{equation*}
Y=\left\{w : w(t)=\sum_{i=0}^{r-1}W_{i}\varphi_{i}(t), W_{i}\in\mathbb{R}^{d}\right\}.
\end{equation*}
Here, $\{\varphi_{i}\}_{i=0}^{r-1}$ are supposed to be
sufficiently smooth and linearly independent on $I$. A function $w$
on $I$ is {called} a piecewise $Y$-type function if for any $0\leq
n\leq N-1$, there exists a function $g\in Y$, such that
$w|_{(t_{n},t_{n+1})}=g|_{(t_{n},t_{n+1})}.$

Sometimes, it is convenient to introduce the transformation
$t=t_{0}+\tau h$ for $\tau\in[0,1]$. Accordingly we denote
$Y_{h}(t_{0})=\left\{v \text{ on } [0,1] : v(\tau)=w(t_{0}+\tau h), w\in
Y\right\}$. Hence
$Y_{h}(t_{0})$=span$\left\{\tilde{\varphi}_{0},\ldots,\tilde{\varphi}_{r-1}\right\}$,
where $\tilde{\varphi}_{i}(\tau)=\varphi_{i}(t_{0}+\tau h)$ for
$i=0,1,\ldots,r-1$. In what follows, lowercase Greek
letters such as $\tau,\sigma,\alpha$ always indicate variables on
the interval [0,1] unless confusions arise.

Given two integrable functions (scalar-valued or vector-valued) $w_{1}$
and $w_{2}$ on $[0,1]$, the inner product
$\langle\cdot,\cdot\rangle$ is defined by
\begin{equation*}
\langle w_{1},w_{2}\rangle=\langle w_{1}(\tau),w_{2}(\tau)\rangle_{\tau}=\int_{0}^{1}w_{1}(\tau)\cdot w_{2}(\tau)d\tau,
\end{equation*}
where $\cdot$ is the entrywise multiplication operation if
$w_{1},w_{2}$ are both vector-valued functions of the same length.
\end{preli}

Given two finite-dimensional function spaces $X$ and $Y$ whose
members are $\mathbb{R}^{d}$-valued, the continuous finite element
method for \eqref{IVP} is described as follows. \par Find a
continuous piecewise X-type function $U(t)$ on $I$ with
$U(t_{0})=y_{0}$, such that
\begin{equation}\label{CFE}
\int_{I}v(t)\cdot(U^{\prime}(t)-f(U(t)))dt=0,
\end{equation}
for any piecewise Y-type function $v(t)$, where $U(t)\approx y(t)$
on $I$ and $y(t)$ solves the IVP \eqref{IVP}. {The term `continuous
finite element'(CFE) comes from the continuity of the finite element
solution $U(t)$.} Since \eqref{CFE} deals with an initial value
problem, we need only to consider it on $[t_{0},t_{0}+h]$.\par Find
$u\in X_{h}(t_{0})$ {with} $u(0)=y_{0}$, such that
\begin{equation}\label{CFE2}
\langle v,u^{\prime}\rangle=h\langle v,f\circ u\rangle,
\end{equation}
for any $v\in Y_{h}(t_{0})$, where $$u(\tau)=U(t_{0}+\tau h)\approx
y(t_{0}+\tau h)$$ for $\tau\in[0,1]$. Since $U(t)$ is continuous,
$y_{1}=u(1)$ is the initial value of the local problem on the next
interval $[t_{1},t_{2}]$. Thus we can solve the global variational
problem \eqref{CFE} on $I$ step by step.

In the special case of
\begin{equation*}
X=\text{span}\left\{1,t,\ldots,t^{r}\right\},\quad Y=\text{span}\left\{1,t,\ldots,t^{r-1}\right\},
\end{equation*}
\eqref{CFE} {reduces to} the classical continuous finite element
method (see, e.g. \cite{Betsch2000,Hulme1972}) {denoted by CFE$r$ in
this paper.} For the purpose of deriving functionally-fitted
methods, we generalize $X$ and $Y$ a little:
\begin{equation}\label{FF}
Y=\text{span}\left\{\varphi_{0}(t),\ldots,\varphi_{r-1}(t)\right\},\quad
X=\text{span}\left\{1,\int_{t_{0}}^{t}\varphi_{0}(s)ds,\ldots,\int_{t_{0}}^{t}\varphi_{r-1}(s)ds\right\}.
\end{equation}
Thus it is
sufficient to give $X$ or $Y$ since they can be determined by each
other. Besides, $Y$ is supposed to be invariant under translation
and reflection , namely,
\begin{equation}\label{affine}
\begin{aligned}
&v(t)\in Y \Rightarrow v(t+c)\in Y \text{    for any   }c\in\mathbb{R},\\
&v(t)\in Y \Rightarrow v(-t)\in Y.\\
\end{aligned}
\end{equation}
Clearly, $Y_{h}(t_{0})$ and $X_{h}(t_{0})$ are irrelevant to $t_{0}$
provided \eqref{affine} holds. For convenience, we simplify
$Y_{h}(t_{0})$ and $X_{h}(t_{0})$ by $Y_{h}$ and $X_{h}$
respectively. In the remainder of this paper, we denote the CFE
method \eqref{CFE} or \eqref{CFE2} based on the general function
spaces \eqref{FF} satisfying the condition \eqref{affine} by
FFCFE$r$.

It is noted that the {FFCFE$r$} method \eqref{CFE2} is defined by a
variational problem and the well-definedness of this problem has not
been confirmed yet. Here we presume the existence and uniqueness of
the solution to \eqref{CFE2}. This assumption will be proved in the
next section. With this premise, we are able to present three
significant properties of the {FFCFE$r$} method. At first, by the
definition of the variational problem \eqref{CFE}, the {FFCFE$r$}
method is functionally fitted with respect to the space $X$.
\begin{mytheo}
The {FFCFE$r$} method \eqref{CFE} solves the IVP \eqref{IVP} whose
solution is a piecewise $X$-type function {without any
error.}
\end{mytheo}

Moreover, the {FFCFE$r$} method is inherently energy preserving
method. The next theorem confirms this point.
\begin{mytheo}\label{EP}
The {FFCFE$r$} method \eqref{CFE2} exactly preserves the Hamiltonian
$H$: $H(y_{1})=H(y_{0})$.
\end{mytheo}
\begin{proof}
Firstly, given a vector $V$, its $i$th entry is denoted by $V_{i}$.
For each function $w\in Y_{h}$, setting
$v(\tau)=w(\tau)\cdot e_{i}\in Y_{h}$ in \eqref{CFE2} leads to
\begin{equation*}
\int_{0}^{1}w(\tau)_{i}u^{\prime}(\tau)_{i}d\tau=h\int_{0}^{1}w(\tau)_{i}f(u(\tau))_{i}d\tau,\quad
i=1,2,\ldots,d,
\end{equation*}
where $e_{i}$ is {the $i$th vector of units.} Thus
\begin{equation}\label{inner}
\begin{aligned}
&\int_{0}^{1}w(\tau)^{\intercal}u^{\prime}(\tau)d\tau=\sum_{i=1}^{d}\int_{0}^{1}w(\tau)_{i}u^{\prime}(\tau)_{i}d\tau\\
&=\sum_{i=1}^{d}h\int_{0}^{1}w(\tau)_{i}f(u(\tau))_{i}d\tau=h\int_{0}^{1}w(\tau)^{\intercal}f(u(\tau))d\tau.\\
\end{aligned}
\end{equation}
Since $u(\tau)\in X_{h}$, $u^{\prime}(\tau)\in Y_{h}$ and $J^{-1}u^{\prime}(\tau)\in Y_{h}$,
by taking $w(\tau)=J^{-1}u^{\prime}(\tau)$ in \eqref{inner}, we have
\begin{equation*}
\begin{aligned}
&H(y_{1})-H(y_{0})\\
&=\int_{0}^{1}\frac{d}{d\tau}H(u(\tau))d\tau
=\int_{0}^{1}u^{\prime}(\tau)^{\intercal}\nabla H(u(\tau))d\tau\\
&=\int_{0}^{1}(J^{-1}u^{\prime}(\tau))^{\intercal}f(u(\tau))d\tau
=h^{-1}\int_{0}^{1}u^{\prime}(\tau)^{\intercal}Ju^{\prime}(\tau)d\tau
=0.\\
\end{aligned}
\end{equation*}
This completes the proof.\end{proof}\qed

The FFCFE$r$ method can also be viewed as a one-step method
$\Phi_{h}: y_{0}\rightarrow y_{1}=u(1)$. It is well known that
reversible methods show a better long-term behaviour than
nonsymmetric ones when applied to reversible differential systems
such as \eqref{IVP} (see, e.g. \cite{Hairer2006}). This fact
motivates the investigation of the symmetry of the FFCFE$r$ method.
Since spaces $X$ and $Y$ satisfy the invariance \eqref{affine},
which is a kind of symmetry, the FFCFE$r$ method is expected to be
symmetric.
\begin{mytheo}\label{symmetry}
The {FFCFE$r$} method \eqref{CFE2} is symmetric provided
\eqref{affine} holds.
\end{mytheo}
\begin{proof}
According to \eqref{affine}, we have $X_{h}=X_{-h}, Y_{h}=Y_{-h}.$
Exchanging $y_{0}\leftrightarrow y_{1}$ and replacing $h$ with $-h$
in \eqref{CFE2} give: $u(0)=y_{1}$, $y_{0}=u(1)$, where
\begin{equation*}
\langle v(\tau),u^{\prime}(\tau)\rangle_{\tau}=-h\langle
v(\tau),f(u(\tau))\rangle_{\tau},\quad u(\tau)\in X_{-h}=X_{h},
\end{equation*}
for {each} $v(\tau)\in Y_{-h}=Y_{h}$. Setting
$u_{1}(\tau)=u(1-\tau)\in X_{h}$ and $\tau\to1-\tau$ leads to
$u_{1}(0)=y_{0}$, $y_{1}=u_{1}(1)$,  where
\begin{equation*}
\langle v_{1}(\tau),u_{1}^{\prime}(\tau)\rangle_{\tau}=h\langle v_{1}(\tau),f(u_{1}(\tau))\rangle_{\tau},
\end{equation*}
for {each} $v_{1}(\tau)=v(1-\tau)\in Y_{h}.$ This method is exactly
the same as \eqref{CFE2}, which means that the {FFCFE$r$} method is
symmetric.
\end{proof}\qed

It is well known that polynomials cannot approximate oscillatory
functions satisfactorily. If the problem \eqref{IVP} has a fixed
frequency $\omega$ which can be evaluated effectively in advance,
then the function space containing the pair $\left\{\sin(\omega t),
\cos(\omega t)\right\}$ seems to be a more promising candidate for
$X$ and $Y$ than the polynomial space. {For example, possible $Y$
spaces for deriving the TF CFE method are}
{\begin{equation}\label{TF1CFE} Y_{1}=\left\{\begin{aligned}
&\text{span}\left\{\cos(\omega t),\sin(\omega t)\right\},r=2,\\
&\text{span}\left\{1,t,\ldots,t^{r-3},\cos(\omega t),\sin(\omega t)\right\},r\geq3,\\
\end{aligned}\right.
\end{equation}}
\begin{equation}\label{TF2CFE}
Y_{2}=\text{span}\left\{ \cos(\omega t),\sin(\omega t),\ldots,\cos(k\omega
t),\sin(k\omega t)\right\}, r=2k,
\end{equation}
and
\begin{equation}\label{TF3CFE}
Y_{3}=\text{span}\left\{1,t,\ldots,t^{p},t\cos(\omega t),t\sin(\omega
t),\ldots,t^{k}\cos(\omega t),t^{k}\sin(\omega t)\right\}.
\end{equation}
{Correspondingly, by equipping the FF CFE method with the space
$Y=Y_{1}, Y_{2}$ or $Y_{3}$}, we obtain three families of TF CFE
methods. According to Theorem \ref{EP} and Theorem \ref{symmetry},
all of them are symmetric energy-preserving methods. To exemplify
this framework of the TF CFE method, in numerical experiments, we
will test the TF CFE method denoted by TFCFE$r$ and TF2CFE$r$ based
on the spaces \eqref{TF1CFE} and \eqref{TF2CFE}. It is noted that
TFCFE2 and TF2CFE2 coincide.

\section{Interpretation as continuous-stage Runge--Kutta methods and the analysis on the algebraic order}\label{CRKK}
An interesting connection between CFE methods and RK-type methods
has been shown in a few papers (see, e.g.
\cite{Bottasso1997,Hulme1972,Tang2012}). Since the RK methods are
dominant in the numerical integration of ODEs, it is meaningful and
useful to transform the FFCFE$r$ method to the corresponding RK-type
method which has been widely and conventionally used in
applications. After the transformation, the FFCFE$r$ method can be
analysed and implemented by standard techniques in ODEs
conveniently. To this end, it is helpful to introduce the projection
operation $P_{h}$. Given a continuous $\mathbb{R}^{d}$-valued
function $w$ on $[0,1]$, its projection onto $Y_{h}$, denoted by
$P_{h}w$, is defined by
\begin{equation}\label{DEF}
\langle v,P_{h}w\rangle=\langle v,w\rangle,\quad\text{for any   } v\in Y_{h}.
\end{equation}
Clearly, $P_{h}w(\tau)$ can be uniquely expressed as a linear
combination of $\{\tilde{\varphi}_{i}(\tau)\}_{i=0}^{r-1}$:
$$P_{h}w(\tau)=\sum_{i=0}^{r-1}U_{i}\tilde{\varphi}_{i}(\tau),\quad U_{i}\in\mathbb{R}^{d}.$$
Taking $v(\tau)=\tilde{\varphi}_{i}(\tau)e_{j}$ in \eqref{DEF} for
$i=0,1,\ldots,r-1$ and $j=1,\ldots,d$, it can be observed that
coefficients $U_{i}$ satisfy the equation
\begin{equation*}
M\otimes I_{d\times d}\left(\begin{array}{c}U_{0}\\ \vdots\\U_{r-1}\end{array}\right)=
\left(\begin{array}{c}\langle \tilde{\varphi}_{0},w\rangle\\ \vdots\\ \langle\tilde{\varphi}_{r-1},w\rangle\end{array}\right),
\end{equation*}
where
\begin{equation*}
M=(\langle\tilde{\varphi}_{i},\tilde{\varphi}_{j}\rangle)_{0\leq i,j\leq r-1}.
\end{equation*}
Since $\left\{\tilde{\varphi}_{i}\right\}_{i=0}^{r-1}$ are linearly
independent for $h>0$, the stiffness matrix $M$ is nonsingular.
Thus the projection can be explicitly expressed by
\begin{equation*}
P_{h}w(\tau)=\langle P_{\tau,\sigma},w(\sigma)\rangle_{\sigma},
\end{equation*}
where
\begin{equation}\label{explicit}
P_{\tau,\sigma}=(\tilde{\varphi}_{0}(\tau),\ldots,\tilde{\varphi}_{r-1}(\tau))M^{-1}(\tilde{\varphi}_{0}(\sigma),\ldots,\tilde{\varphi}_{r-1}(\sigma))^{\intercal}.
\end{equation}
Clearly, $P_{\tau,\sigma}$ can be calculated by a basis other than
$\left\{\tilde{\varphi}_{i}\right\}_{i=0}^{r-1}$ since they only differ
in a linear transformation. If
$\left\{\phi_{0},\ldots,\phi_{r-1}\right\}$ is an orthonormal basis
of $X_{h}$ {under the inner product $\langle\cdot,\cdot\rangle$,}
then $P_{\tau,\sigma}$ admits a simpler expression:
\begin{equation}\label{PCOEFF}
P_{\tau,\sigma}=\sum_{i=0}^{r-1}\phi_{i}(\tau)\phi_{i}(\sigma).
\end{equation}
Now using \eqref{CFE2} and the definition \eqref{DEF}
of the operator $P_{h}$, we obtain that $u^{\prime}=hP_{h}(f\circ
u)$ and
\begin{equation}\label{projection}
u^{\prime}(\tau)=h\langle P_{\tau,\sigma},f(u(\sigma))\rangle_{\sigma}.
\end{equation}
Integrating the above equation with respect to $\tau$, we transform
the {FFCFE$r$} method \eqref{CFE2} into the {continuous-stage} RK
method:
\begin{equation}\label{CRK}
\left\{\begin{aligned}
&u(\tau)=y_{0}+h\int_{0}^{1}A_{\tau,\sigma}f(u(\sigma))d\sigma,\\
&y_{1}=u(1),\\
\end{aligned}\right.
\end{equation}
where
\begin{equation}\label{Aexplicit}
A_{\tau,\sigma}=\int_{0}^{\tau}P_{\alpha,\sigma}d\alpha=\sum_{i=0}^{r-1}\int_{0}^{\tau}\phi_{i}(\alpha)d\alpha\phi_{i}(\sigma).
\end{equation}
In particular,
\begin{equation}\label{Legendre}
\phi_{i}(\tau)=\hat{p}_{i}(\tau),
\end{equation}
for the CFE$r$ method for $i=0,1,\ldots,r-1$, {where
$\hat{p}_{i}(\tau)$ is the shifted Legendre polynomial of degree $i$
on $[0,1]$, scaled in order to be orthonormal. Thus} the CFE$r$
method in the form \eqref{CRK} is identical to the energy-preserving
collocation method of order $2r$ (see \cite{Hairer2010}) or the
Hamiltonian boundary value method HBVM$(\infty,r)$ (see, e.g.
\cite{Brugnano2010}). For the {FFCFE$r$} method, since
$P_{\tau,\sigma}, A_{\tau,\sigma}$ are functions of variable $h$ and
$u(\tau)$ is implicitly {determined} by \eqref{CRK}, it is necessary
to analyse their smoothness with respect to $h$ before
{investigating the analytic property} of the numerical solution
$u(\tau)$. First of all, it can be observed from \eqref{explicit}
that $P_{\tau,\sigma}=P_{\tau,\sigma}(h)$ is not defined at $h=0$
since the matrix $M$ is singular {in this case.} Fortunately, the
following lemma shows that { the singularity is removable}.
\begin{lem}\label{smooth}
When $h$ tends to 0, the limit of $P_{\tau,\sigma}$ there exists.
What's more, $P_{\tau,\sigma}$ can be smoothly extended to $h=0$ by
setting $P_{\tau,\sigma}(0)=\lim_{h\to0}P_{\tau,\sigma}(h)$.
\end{lem}
\begin{proof}
By expanding $\{\varphi_{i}(t_{0}+\tau h)\}_{i=0}^{r-1}$ at $t_{0}$, we obtain that
\begin{equation}\label{expand1}
(\tilde{\varphi}_{0}(\tau),\ldots,\tilde{\varphi}_{r-1}(\tau))=(1,\tau h,\ldots,\frac{\tau^{r-1}h^{r-1}}{(r-1)!})W+\mathcal{O}(h^{r}),
\end{equation}
{where
\begin{equation}\label{Wronskian}
W=\left(\begin{array}{cccc}\varphi_{0}(t_{0})&\varphi_{1}(t_{0})&\cdots&\varphi_{r-1}(t_{0})\\
\varphi_{0}^{(1)}(t_{0})&\varphi_{1}^{(1)}(t_{0})&\cdots&\varphi_{r-1}^{(1)}(t_{0})\\ \vdots&\vdots& &\vdots\\
\varphi_{0}^{(r-1)}(t_{0})&\varphi_{1}^{(r-1)}(t_{0})&\cdots&\varphi_{r-1}^{(r-1)}(t_{0})\end{array}\right)
\end{equation}}
is the Wronskian of $\{\varphi_{i}(t)\}_{i=0}^{r-1}$ at $t_{0}$,
which is nonsingular. {Post-multiplying the right-hand side of
\eqref{expand1} by} $W^{-1}diag(1,h^{-1},\ldots,h^{1-r}(r-1)!)$
yields another basis of $X_{h}$:
\begin{equation*}
\{1+\mathcal{O}(h),\tau+\mathcal{O}(h),\ldots,\tau^{r-1}+\mathcal{O}(h)\}.
\end{equation*}
Applying the Gram-Schmidt process (with respect to the inner product
$\langle\cdot,\cdot\rangle$) to the above basis, we obtain an
orthonormal basis
$\left\{\phi_{i}(\tau)=\hat{p}_{i}(\tau)+\mathcal{O}(h)\right\}_{i=0}^{r-1}$.
Thus by \eqref{PCOEFF} and defining
\begin{equation}\label{Plimit}
P_{\tau,\sigma}(0)=\lim_{h\to0}\sum_{i=0}^{r-1}\phi_{i}(\tau)\phi_{i}(\sigma)
=\sum_{i=0}^{r-1}\hat{p}_{i}(\tau)\hat{p}_{i}(\sigma),
\end{equation}
$P_{\tau,\sigma}$ is
extended to $h=0$. Since each
$\phi_{i}(\tau)=\hat{p}_{i}(\tau)+\mathcal{O}(h)$ is smooth with respect
to $h$, $P_{\tau,\sigma}$ is also a smooth function of $h$.
\end{proof}\qed

{From \eqref{Legendre} and \eqref{Plimit}, it can be observed that
the FFCFE$r$ method \eqref{CRK} reduces to the CFE$r$ method when
$h\to0$,} or equivalently, the energy-preserving collocation method
of order $2r$ and HBVM$(\infty,r)$ method mentioned above. Since
$A_{\tau,\sigma}=\int_{0}^{\tau}P_{\alpha,\sigma}d\alpha$ is also a
smooth function of $h$, we can assume that
\begin{equation}\label{cond1}
M_{k}=\max_{\tau,\sigma,h\in[0,1]}\left|\frac{\partial^{k}A_{\tau,\sigma}}{\partial h^{k}}\right|,\quad k=0,1,\ldots.
\end{equation}
Furthermore, {since the right function $f$ in \eqref{IVP} maps from
$\mathbb{R}^{d}$ to $\mathbb{R}^{d}$, the $n$th-order derivative of
$f$ at $y$ denoted by $f^{(n)}(y)$} is a multilinear map from
$\underbrace{\mathbb{R}^{d}\times\ldots\times\mathbb{R}^{d}}_{n-fold}$
to $\mathbb{R}^{d}$ defined by {\begin{equation*}
f^{(n)}(y)(z_{1},\ldots,z_{n})=\sum_{1\leq\alpha_{1},\ldots,\alpha_{n}\leq
d}\frac{\partial^{n}f}{\partial y_{\alpha_{1}}\cdots\partial
y_{\alpha_{n}}}(y)z_{1}^{\alpha_{1}}\ldots z_{n}^{\alpha_{n}},
\end{equation*}
where $y=(y_{1},\ldots,y_{d})^{\intercal}$} and
$z_{i}=(z_{i}^{1},\ldots,z_{i}^{d})^{\intercal}$, $i=1,\ldots,n$.
With this background, we now  can give the existence, uniqueness,
especially the smoothness with respect to $h$ for the continuous
finite element {approximation} $u(\tau)$ associated with the
FFCFE$r$ method. The proof of the following theorem is based on a
fixed-point iteration which is analogous to Picard iteration.
\begin{mytheo}\label{eus}
{Given a positive constant $R$,} let
$$B(y_{0},R)=\left\{y\in\mathbb{R}^{d} : ||y-y_{0}||\leq R\right\}$$ and
\begin{equation}\label{cond}
D_{n}=\max_{y\in B(y_{0},R)}||f^{(n)}(y)||,\quad n=0,1,\ldots,
\end{equation}
where $||\cdot||=||\cdot||_{\infty}$ is the maximum norm for vectors
in $\mathbb{R}^{d}$ or the corresponding induced norm for the
multilinear maps $f^{(n)}(y), n\geq1$. Then the {FFCFE$r$} method
\eqref{CFE2} or \eqref{CRK} has a unique solution $u(\tau)$ which is
{smoothly} dependent of $h$ provided {\begin{equation}\label{cond2}
0\leq
h\leq\varepsilon<\min\left\{\frac{1}{M_{0}D_{1}},\frac{R}{M_{0}D_{0}},1\right\}.
\end{equation}}
\end{mytheo}
\begin{proof}
Set $u_{0}(\tau)\equiv y_{0}$. We construct a function
series $\{u_{n}(\tau)\}_{n=0}^{\infty}$ defined by the relation
\begin{equation}\label{recur}
u_{n+1}(\tau)=y_{0}+h\int_{0}^{1}A_{\tau,\sigma}f(u_{n}(\sigma))d\sigma,{\quad
n=0,1,\ldots.}
\end{equation}
Obviously, $\lim_{n\to\infty}u_{n}(\tau)$ is a solution to
\eqref{CRK} provided $\left\{u_{n}(\tau)\right\}_{n=0}^{\infty}$ is uniformly
convergent. Thus we need only to prove the uniform
convergence of the infinite series
$$\sum_{n=0}^{\infty}(u_{n+1}(\tau)-u_{n}(\tau)).$$

It follows from \eqref{cond1}, \eqref{cond2}, \eqref{recur} and
induction that
\begin{equation}\label{cond3}
||u_{n}(\tau)-y_{0}||\leq R,\quad n=0,1,\ldots
\end{equation}
Then {by} using \eqref{cond}, \eqref{cond2}, \eqref{recur},
\eqref{cond3} and the inequalities
\begin{equation*}
\begin{aligned}
&||\int_{0}^{1}w(\tau)d\tau||\leq\int_{0}^{1}||w(\tau)||d\tau,\quad\text{for $\mathbb{R}^{d}$-valued function $w(\tau)$},\\
&{||f(y)-f(z)||\leq D_{1}||y-z||,\quad\text{for } y,z\in B(y_{0},R),}\\
\end{aligned}
\end{equation*}
we obtain the
following inequalities
\begin{equation*}
\begin{aligned}
&||u_{n+1}(\tau)-u_{n}(\tau)||\leq h\int_{0}^{1}M_{0}D_{1}||u_{n}(\sigma)-u_{n-1}(\sigma)||d\sigma\\
&\leq\beta||u_{n}-u_{n-1}||_{c},\quad{\beta=\varepsilon M_{0}D_{1},}\\
\end{aligned}
\end{equation*}
where $||\cdot||_{c}$ is the maximum norm for continuous functions:
\begin{equation*}
||w||_{c}=\max_{\tau\in[0,1]}||w(\tau)||,\quad\text{$w$ is a continuous $\mathbb{R}^{d}$-valued function on $[0,1]$}.
\end{equation*}
Thus, we have
\begin{equation*}
||u_{n+1}-u_{n}||_{c}\leq\beta||u_{n}-u_{n-1}||_{c}
\end{equation*}
and
\begin{equation}\label{recur3}
||u_{n+1}-u_{n}||_{c}\leq\beta^{n}||u_{1}-y_{0}||_{c}\leq\beta^{n}R,\quad n=0,1,\ldots.
\end{equation}
Since $\beta<1$, according to Weierstrass $M$-test,
$\sum_{n=0}^{\infty}(u_{n+1}(\tau)-u_{n}(\tau))$ is uniformly
convergent, and thus, the limit of $\{u_{n}(\tau)\}_{n=0}^{\infty}$
is a solution to \eqref{CRK}. If $v(\tau)$ is another solution, then
the difference {between} $u(\tau)$ and $v(\tau)$ satisfies
{\begin{equation*} ||u(\tau)-v(\tau)||\leq
h\int_{0}^{1}||A_{\tau,\sigma}(f(u(\sigma))-f(v(\sigma)))||d\sigma\leq\beta||u-v||_{c},
\end{equation*}
and
\begin{equation*}
||u-v||_{c}\leq\beta||u-v||_{c}.
\end{equation*}}
This means $||u-v||_{c}=0,$ i.e., $u(\tau)\equiv v(\tau)$. Hence the
existence and uniqueness have been proved.

As for the {smooth} dependence of $u$ on $h$, since every
$u_{n}(\tau)$ is a {smooth} function of $h$, we need only to prove
the series
$$\left\{\frac{\partial^{k}u_{n}}{\partial h^{k}}(\tau)\right\}_{n=0}^{\infty}$$
are uniformly convergent for {$k\geq1$.} Firstly, differentiating
both sides of \eqref{recur} with respect to $h$ yields
\begin{equation}\label{recur2}
\frac{\partial u_{n+1}}{\partial h}(\tau)=\int_{0}^{1}(A_{\tau,\sigma}+h\frac{\partial A_{\tau,\sigma}}{\partial h})f(u_{n}(\sigma))d\sigma
+h\int_{0}^{1}A_{\tau,\sigma}f^{(1)}(u_{n}(\sigma))\frac{\partial u_{n}}{\partial h}(\sigma)d\sigma.
\end{equation}
We then have
\begin{equation}
||\frac{\partial u_{n+1}}{\partial h}||_{c}\leq\alpha+\beta||\frac{\partial u_{n}}{\partial h}||_{c},\quad
\alpha=(M_{0}+\varepsilon M_{1})D_{0}.
\end{equation}
By induction, it is easy to show that $\left\{\frac{\partial
u_{n}}{\partial h}(\tau)\right\}_{n=0}^{\infty}$ is uniformly
bounded:
\begin{equation}\label{bound}
||\frac{\partial u_{n}}{\partial
h}||_{c}\leq\alpha(1+\beta+\ldots+\beta^{n-1})\leq\frac{\alpha}{1-\beta}=C^{*},\quad
n=0,1,\ldots.
\end{equation}
Combining \eqref{recur3}, \eqref{recur2} and \eqref{bound}, we
obtain
\begin{equation*}
\begin{aligned}
&||\frac{\partial u_{n+1}}{\partial h}-\frac{\partial u_{n}}{\partial h}||_{c}\\
&\leq\int_{0}^{1}(M_{0}+hM_{1})||f(u_{n}(\sigma))-f(u_{n-1}(\sigma))||d\sigma\\
&+h\int_{0}^{1}M_{0}\left(||(f^{(1)}(u_{n}(\sigma))-f^{(1)}(u_{n-1}(\sigma)))\frac{\partial u_{n}}{\partial h}(\sigma)||
+||f^{(1)}(u_{n-1}(\sigma))(\frac{\partial u_{n}}{\partial h}(\sigma)-\frac{\partial u_{n-1}}{\partial h}(\sigma))||\right)d\sigma\\
&\leq\gamma\beta^{n-1}+\beta||\frac{\partial u_{n}}{\partial h}-\frac{\partial u_{n-1}}{\partial h}||_{c},\\
\end{aligned}
\end{equation*}
where
$$\gamma=(M_{0}D_{1}+\varepsilon M_{1}D_{1}+\varepsilon M_{0}L_{2}C^{*})R,$$
{and $L_{2}$ is a constant satisfying
\begin{equation*}
||f^{(1)}(y)-f^{(1)}(z)||\leq L_{2}||y-z||,\quad\text{for $y,z\in B(y_{0},R)$}.
\end{equation*}}
Thus again by induction, we have
\begin{equation*}
||\frac{\partial u_{n+1}}{\partial h}-\frac{\partial u_{n}}{\partial h}||_{c}\leq n\gamma\beta^{n-1}+\beta^{n}C^{*},\quad n=1,2,\ldots
\end{equation*}
and $\left\{\frac{\partial u_{n}}{\partial
h}(\tau)\right\}_{n=0}^{\infty}$ is uniformly convergent. By a
similar argument, one can show that other function series
$\left\{\frac{\partial^{k}u_{n}}{\partial
h^{k}}(\tau)\right\}_{n=0}^{\infty}$ for $k\geq2$ are uniformly
convergent as well. Therefore, $u(\tau)$ is {smoothly dependent} on
$h$. The proof is complete.
\end{proof}\qed

Since our analysis on the algebraic order of the {FFCFE$r$} method
is mainly based on Taylor's theorem, it is meaningful to investigate
the expansion of $P_{\tau,\sigma}$.
\begin{prop}\label{prop1}
Assume that the Taylor expansion of $P_{\tau,\sigma}$ with respect
to $h$ at zero is
\begin{equation}\label{Pexpand}
P_{\tau,\sigma}=\sum_{n=0}^{r-1}P_{\tau,\sigma}^{[n]}h^{n}+\mathcal{O}(h^{r}).
\end{equation}
Then the coefficients $P_{\tau,\sigma}^{[n]}$ satisfy
\begin{equation*}
\langle P_{\tau,\sigma}^{[n]},g_{m}(\sigma)\rangle_{\sigma}=
\left\{\begin{aligned}
&g_{m}(\tau) ,{\quad n=0,\quad m=r-1,}\\
&0,\quad{n=1,\ldots,r-1,\quad m=r-1-n,}\\
\end{aligned}\right.
\end{equation*}
for any $g_{m}\in P_{m}([0,1]),$ {where} $P_{m}([0,1])$ consists of
polynomials of degrees $\leq m$ on $[0,1]$.
\end{prop}
\begin{proof}
It can be observed from \eqref{explicit} that
\begin{equation}\label{equa}
\langle P_{\tau,\sigma},\varphi_{i}(t_{0}+\sigma h)\rangle_{\sigma}=\varphi_{i}(t_{0}+\tau h),\quad i=0,1,\ldots,r-1.
\end{equation}
Meanwhile, expanding $\varphi_{i}(t_{0}+\tau
h)$ at $t_{0}$ yields
\begin{equation}\label{expand}
\varphi_{i}(t_{0}+\tau h)=\sum_{n=0}^{r-1}\frac{\varphi_{i}^{(n)}(t_{0})}{n!}\tau^{n}h^{n}+\mathcal{O}(h^{r}).
\end{equation}
Then by inserting \eqref{Pexpand} and \eqref{expand} into the equation \eqref{equa}, we
obtain that
\begin{equation*}
\langle\sum_{n=0}^{r-1}P_{\tau,\sigma}^{[n]}h^{n},\sum_{m=0}^{r-1}\frac{\varphi_{i}^{(m)}(t_{0})}{m!}\sigma^{m}h^{m}\rangle_{\sigma}
=\sum_{k=0}^{r-1}\frac{\varphi_{i}^{(k)}(t_{0})}{k!}\tau^{k}h^{k}+\mathcal{O}(h^{r}).
\end{equation*}
Collecting the terms by $h^{k}$ leads to
\begin{equation*}
\begin{aligned}
&\sum_{k=0}^{r-1}(\sum_{m+n=k}\frac{\varphi_{i}^{(m)}(t_{0})}{m!}
\langle P_{\tau,\sigma}^{[n]},\sigma^{m}\rangle_{\sigma}-\frac{\varphi_{i}^{(k)}(t_{0})}{k!}\tau^{k})h^{k}=\mathcal{O}(h^{r}),\\
&\sum_{m=0}^{k-1}\frac{\varphi_{i}^{(m)}(t_{0})}{m!}
P_{m,k-m}+\frac{\varphi_{i}^{(k)}(t_{0})}{k!}(P_{k0}-\tau^{k})=0,\quad i,\ k=0,1,\ldots,r-1,\\
\end{aligned}
\end{equation*}
and
\begin{equation*}
W^{\intercal} V=0,
\end{equation*}
where $P_{mn}=\langle
P_{\tau,\sigma}^{[n]},\sigma^{m}\rangle_{\sigma},\ W$ is the
Wronskian \eqref{Wronskian}, and $V=(V_{mk})_{0\leq m,k\leq r-1}$ is
an upper triangular matrix with the entries determined
by
\begin{equation*}
V_{mk}=\left\{\begin{aligned}
&\frac{1}{m!}P_{m,k-m},\quad m<k,\\
&\frac{1}{m!}(P_{m,0}-\tau^{m}),\quad m=k.\\
\end{aligned}\right.
\end{equation*}
Since $W$ is nonsingular, $V=0$,
\begin{equation}\label{result}
P_{mn}=\left\{\begin{aligned}
&\tau^{m},\quad n=0,\quad m+n\leq r-1,\\
&0,\quad n=1, 2, \ldots, r-1,\quad m+n\leq r-1.\\
\end{aligned}\right.
\end{equation}
{Then the statement of this proposition directly follows from
\eqref{result}.}
\end{proof}\qed

Aside from $P_{\tau,\sigma}$, it is also crucial to analyse the
expansion of the solution $u(\tau)$. For convenience, we say that an
$h$-dependent function $w(\tau)$ is regular if it can be expanded as
\begin{equation*}
w(\tau)=\sum_{n=0}^{r-1}w^{[n]}(\tau)h^{n}+\mathcal{O}(h^{r}),
\end{equation*}
where
$$w^{[n]}(\tau)=\frac{1}{n!}\frac{\partial^{n}w(\tau)}{\partial h^{n}}|_{h=0}$$ is a vector-valued function
with polynomial entries of degrees $\leq n$.
\begin{lem}\label{lemma1}
Given a regular function $w$ and an $h$-independent {sufficiently
smooth} function $g$, the composition (if exists) is
regular. Moreover, {the difference between $w$ and its projection
satisfies}
\begin{equation*}
P_{h}w(\tau)-w(\tau)=\mathcal{O}(h^{r}).
\end{equation*}
\end{lem}
\begin{proof}
Assume that the expansion of $g(w(\tau))$ with respect to $h$ at
zero is
\begin{equation*}
g(w(\tau))=\sum_{n=0}^{r-1}p^{[n]}(\tau)h^{n}+\mathcal{O}(h^{r}).
\end{equation*}
Then by differentiating $g(w(\tau))$ with respect to $h$ at zero iteratively and using
$$p^{[n]}(\tau)=\frac{1}{n!}\frac{\partial^{n}g(w(\tau))}{\partial h^{n}}|_{h=0},\quad\text{the degree of $\frac{\partial^{n}w(\tau)}{\partial h^{n}}|_{h=0}\leq n$},\quad {n=0,1,\ldots,r-1,}$$
it can be observed that $p^{[n]}(\tau)$ is a vector with polynomials entries of degrees $\leq n$ for $n=0,1,\ldots,r-1$ and the first statement is confirmed.

As for the second statement,
using Proposition \ref{prop1}, we have
\begin{equation*}
\begin{aligned}
&P_{h}w(\tau)-w(\tau)\\
&=\langle\sum_{n=0}^{r-1}P_{\tau,\sigma}^{[n]}h^{n},\sum_{k=0}^{r-1}w^{[k]}(\sigma)h^{k}\rangle_{\sigma}-\sum_{m=0}^{r-1}w^{[m]}(\tau)h^{m}+\mathcal{O}(h^{r})\\
&=\sum_{m=0}^{r-1}(\sum_{n+k=m}\langle P_{\tau,\sigma}^{[n]},w^{[k]}(\sigma)\rangle_{\sigma}-w^{[m]}(\tau))h^{m}+\mathcal{O}(h^{r})\\
&=\sum_{m=0}^{r-1}(\langle P_{\tau,\sigma}^{[0]},w^{[m]}(\sigma)\rangle_{\sigma}-w^{[m]}(\tau))h^{m}+\mathcal{O}(h^{r})=\mathcal{O}(h^{r}).\\
\end{aligned}
\end{equation*}
\end{proof}\qed

Before going further, it may be useful to recall some standard
results in the theory of ODEs. To emphasize the dependence of the
solution to $y^{\prime}(t)=f(y(t))$ on the initial value, we assume
that $y(\cdot,\tilde{t},\tilde{y})$ solves the IVP :
\begin{equation*}
\frac{d}{dt}y(t,\tilde{t},\tilde{y})=f(y(t,\tilde{t},\tilde{y})),\quad y(\tilde{t},\tilde{t},\tilde{y})=\tilde{y}.
\end{equation*}
Clearly, this problem is equivalent to the following
integral equation
$$y(t,\tilde{t},\tilde{y})=\tilde{y}+\int_{\tilde{t}}^{t}f(y(\xi,\tilde{t},\tilde{y}))d\xi.$$
Differentiating it with respect to $\tilde{t}$ and $\tilde{y}$ and using the uniqueness of the solution results in
\begin{equation}\label{variation}
\frac{\partial y}{\partial\tilde{t}}(t,\tilde{t},\tilde{y})=-\frac{\partial y}{\partial\tilde{y}}(t,\tilde{t},\tilde{y})f(\tilde{y}).
\end{equation}
With the previous analysis results, we are in the position to give
the order of FFCFE$r$.

\begin{mytheo}\label{order}
The stage order and order of the {FFCFE$r$} method \eqref{CFE2} or
\eqref{CRK} are $r$ and $2r$ respectively, that is,
\begin{equation*} u(\tau)-y(t_{0}+\tau
h)=\mathcal{O}(h^{r+1}),
\end{equation*}
for $0<\tau<1$, and
\begin{equation*}
\quad u(1)-y(t_{0}+h)=\mathcal{O}(h^{2r+1}).
\end{equation*}
\end{mytheo}
\begin{proof}
Firstly, by Theorem \ref{eus} and Lemma \ref{smooth}, we can expand
$u(\tau)$ and $A_{\tau,\sigma}$ with respect to $h$ at zero:
\begin{equation*}
u(\tau)=\sum_{m=0}^{r-1}u^{[m]}(\tau)h^{m}+\mathcal{O}(h^{r}),\quad A_{\tau,\sigma}=\sum_{m=0}^{r-1}A_{\tau,\sigma}^{[m]}h^{m}+\mathcal{O}(h^{r}).
\end{equation*}
Then let
\begin{equation*}
\delta=u(\sigma)-y_{0}=\sum_{m=1}^{r-1}u^{[m]}(\sigma)h^{m}+\mathcal{O}(h^{r})=\mathcal{O}(h).
\end{equation*}
Expanding $f(u(\sigma))$ at $y_{0}$ and inserting the above equalities into the
first equation of \eqref{CRK}, we obtain
\begin{equation}\label{comp}
\sum_{m=0}^{r-1}u^{[m]}(\tau)h^{m}=
y_{0}+h\int_{0}^{1}\sum_{k=0}^{r-1}A_{\tau,\sigma}^{[k]}h^{k}\sum_{n=0}^{r-1}F^{(n)}(y_{0})(\underbrace{\delta,\ldots,\delta}_{n-fold})d\sigma+\mathcal{O}(h^{r}),
\end{equation}
where $F^{(n)}(y_{0})=f^{(n)}(y_{0})/n!.$ We claim that $u(\tau)$
is regular, i.e. $u^{[m]}(\tau)\in
P_{m}^{d}=\underbrace{P_{m}([0,1])\times\ldots\times
P_{m}([0,1])}_{d-fold}$ for $m=0,1,\ldots,r-1$. This fact
can be confirmed by induction. Clearly, $u^{[0]}(\tau)=y_{0}\in
P_{0}^{d}$. If $u^{[n]}(\tau)\in P_{n}^{d}, n=0,1,\ldots,m$, then by
comparing the coefficients of $h^{m+1}$ on both sides of
\eqref{comp} and using \eqref{Aexplicit} and Proposition
\ref{prop1}, we obtain that
\begin{equation*}
\begin{aligned}
&u^{[m+1]}(\tau)=
\sum_{k+n=m}\int_{0}^{1}A_{\tau,\sigma}^{[k]}g_{n}(\sigma)d\sigma=\sum_{k+n=m}\int_{0}^{\tau}\int_{0}^{1}P_{\alpha,\sigma}^{[k]}g_{n}(\sigma)d\sigma d\alpha\\
&=\int_{0}^{\tau}\int_{0}^{1}P_{\alpha,\sigma}^{[0]}g_{m}(\sigma)d\sigma d\alpha
=\int_{0}^{\tau}g_{m}(\alpha)d\alpha\in P_{m+1}^{d},\quad g_{n}(\sigma)\in P_{n}^{d}.\\
\end{aligned}
\end{equation*}
This completes the induction. By Lemma \ref{lemma1}, $f(u(\tau))$
is also regular and
\begin{equation}\label{error}
f(u(\tau))-P_{h}(f\circ u)(\tau)=\mathcal{O}(h^{r}).
\end{equation}
Then it follows from \eqref{projection}, \eqref{variation} and \eqref{error} that
\begin{equation}\label{stage}
\begin{aligned}
&u(\tau)-y(t_{0}+\tau h)=y(t_{0}+\tau h,t_{0}+\tau h,u(\tau))-y(t_{0}+\tau h,t_{0},y_{0})\\
&=\int_{0}^{\tau}\frac{d}{d\alpha}y(t_{0}+\tau h,t_{0}+\alpha h,u(\alpha))d\alpha\\
&=\int_{0}^{\tau}(h\frac{\partial y}{\partial\tilde{t}}(t_{0}+\tau h,t_{0}+\alpha h,u(\alpha))+
\frac{\partial y}{\partial\tilde{y}}(t_{0}+\tau h,t_{0}+\alpha h,u(\alpha))u^{\prime}(\alpha))d\alpha\\
&=-h\int_{0}^{\tau}\Phi^{\tau}(\alpha)(f(u(\alpha))-P_{h}(f\circ u)(\alpha))d\alpha\\
&=\mathcal{O}(h^{r+1}),\\
\end{aligned}
\end{equation}
where
$$\Phi^{\tau}(\alpha)=\frac{\partial y}{\partial\tilde{y}}(t_{0}+\tau h,t_{0}+\alpha h,u(\alpha)).$$
As for the algebraic order, setting
$\tau=1$ in \eqref{stage} leads to
\begin{equation}\label{inter}
\begin{aligned}
&u(1)-y(t_{0}+h)\\
&=-h\int_{0}^{1}\Phi^{1}(\alpha)(f(u(\alpha))-P_{h}(f\circ u)(\alpha))d\alpha.\\
\end{aligned}
\end{equation}
Since $\Phi^{1}(\alpha)$ is a matrix-valued function, we partition it
as
$\Phi^{1}(\alpha)=(\Phi_{1}^{1}(\alpha),\ldots,\Phi_{d}^{1}(\alpha))^{\intercal}$.
Using Lemma \ref{lemma1} again leads to
\begin{equation}\label{pre1}
\Phi_{i}^{1}(\alpha)=P_{h}\Phi_{i}^{1}(\alpha)+\mathcal{O}(h^{r}),\quad i=1,2,\ldots,d.
\end{equation}
Meanwhile, setting $w(\alpha)=P_{h}\Phi_{i}(\alpha)^{\intercal}$ in
\eqref{inner} and using \eqref{projection} yield
\begin{equation}\label{pre2}
\int_{0}^{1}P_{h}\Phi_{i}^{1}(\alpha)f(u(\alpha))d\alpha=h^{-1}\int_{0}^{1}P_{h}\Phi_{i}^{1}(\alpha)u^{\prime}(\alpha)d\alpha
=\int_{0}^{1}P_{h}\Phi_{i}^{1}(\alpha)P_{h}(f\circ u)(\alpha)d\alpha,\quad i=1,2,\ldots,d.
\end{equation}
Therefore, using \eqref{inter}, \eqref{pre1} and \eqref{pre2} we have
\begin{equation*}
\begin{aligned}
&u(1)-y(t_{0}+h)\\
&=-h\int_{0}^{1}\left(\left(\begin{array}{c}P_{h}\Phi_{1}^{1}(\alpha)\\ \vdots\\P_{h}\Phi_{d}^{1}(\alpha)\end{array}\right)+\mathcal{O}(h^{r})\right)(f(u(\alpha))-P_{h}(f\circ u)(\alpha))d\alpha\\
&=-h\int_{0}^{1}\left(\begin{array}{c}P_{h}\Phi_{1}^{1}(\alpha)(f(u(\alpha))-P_{h}(f\circ u)(\alpha))\\ \vdots\\P_{h}\Phi_{d}^{1}(\alpha)(f(u(\alpha))-P_{h}(f\circ u)(\alpha))\end{array}\right)d\alpha-h\int_{0}^{1}\mathcal{O}(h^{r})\times\mathcal{O}(h^{r})d\alpha=\mathcal{O}(h^{2r+1}).\\
\end{aligned}
\end{equation*}
\end{proof}\qed

{According to Theorem \ref{order}, the TF CFE methods based on the
spaces \eqref{TF1CFE}, \eqref{TF2CFE} and \eqref{TF3CFE} are of
order $2r$, $4k$ and $2(k+p+1)$, respectively.}

\section{Implementation issues}\label{IMPLE}
It should be noted that \eqref{CRK} is not a practical form for
applications. In this section, we will detail the implementation of
the {FFCFE$r$} method. Firstly, it is necessary to introduce the
generalized Lagrange interpolation functions $l_{i}(\tau)\in X_{h}$
with respect to $(r+1)$ distinct points
$\{d_{i}\}_{i=1}^{r+1}\subseteq[0,1]$:
\begin{equation}\label{Lagrange}
(l_{1}(\tau),\ldots,l_{r+1}(\tau))=(\widetilde{\Phi}_{1}(\tau),\widetilde{\Phi}_{2}(\tau),\ldots,\widetilde{\Phi}_{r+1}(\tau))\Lambda^{-1},
\end{equation}
where {$\{\Phi_{i}(t)\}_{i=1}^{r+1}$ is a basis of $X$,
$\widetilde{\Phi}_{i}(\tau)=\Phi_{i}(t_{0}+\tau h)$} and
\begin{equation*}
\Lambda=\left(\begin{array}{cccc}\widetilde{\Phi}_{1}(d_{1})&\widetilde{\Phi}_{2}(d_{1})&\ldots&\widetilde{\Phi}_{r+1}(d_{1})\\
\widetilde{\Phi}_{1}(d_{2})&\widetilde{\Phi}_{2}(d_{2})&\ldots&\widetilde{\Phi}_{r+1}(d_{2})\\ \vdots&\vdots& &\vdots\\
\widetilde{\Phi}_{1}(d_{r+1})&\widetilde{\Phi}_{2}(d_{r+1})&\ldots&\widetilde{\Phi}_{r+1}(d_{r+1})\\\end{array}\right).
\end{equation*}
By {means of the} expansions
\begin{equation*}
\Phi_{i}(t_{0}+d_{j}h)=\sum_{n=0}^{r}\frac{\Phi_{i}^{(n)}(t_{0})}{n!}d_{j}^{n}h^{n}+\mathcal{O}(h^{r+1}),\quad
i,j=1,2,\ldots,r+1,
\end{equation*}
we have
\begin{equation*}
\Lambda=\left(\begin{array}{cccc}1&d_{1}h&\ldots&\frac{d_{1}^{r}h^{r}}{r!}\\
1&d_{2}h&\ldots&\frac{d_{2}^{r}h^{r}}{r!}\\
\vdots&\vdots& &\vdots\\
1&d_{r+1}h&\ldots&\frac{d_{r+1}^{r}h^{r}}{r!}\\\end{array}\right)\widetilde{W}+\mathcal{O}(h^{r+1}),
\end{equation*}
where $\widetilde{W}$ is the Wronskian of
$\{\Phi_{i}(t)\}_{i=1}^{r+1}$ at $t_{0}$. Since $\widetilde{W}$ is nonsingular, $\Lambda$ is also nonsingular
for $h$ which is sufficiently small but not zero and the equation
\eqref{Lagrange} makes sense in this case. Then
$\{l_{i}(\tau)\}_{i=1}^{r+1}$ is a basis of $X_{h}$ satisfying
$l_{i}(d_{j})=\delta_{ij}$ and $u(\tau)$ can be expressed as
\begin{equation*}
u(\tau)=\sum_{i=1}^{r+1}u(d_{i})l_{i}(\tau).
\end{equation*}
Choosing
$d_{i}=(i-1)/r$ and denoting $y_{\sigma}=u(\sigma)$,
\eqref{CRK} now reads
\begin{equation}\label{PCRK}
\left\{\begin{aligned}
&y_{\sigma}=\sum_{i=1}^{r+1}y_{\frac{i-1}{r}}l_{i}(\sigma),\\
&y_{\frac{i-1}{r}}=y_{0}+h\int_{0}^{1}A_{\frac{i-1}{r},\sigma}f(y_{\sigma})d\sigma,\quad i=2,\ldots,r+1.\\
\end{aligned}\right.
\end{equation}
When $f$ is a polynomial and $\{\Phi_{i}(t)\}_{i=1}^{r+1}$ are
polynomials, trigonometrical or exponential functions, the integral
in \eqref{PCRK} can be calculated exactly. After solving this
algebraic system about variables $y_{1/r}, y_{2/r},\ldots,y_{1}$ by
iterations, we obtain the numerical solution $y_{1}\approx
y(t_{0}+h)$. Therefore, although the FFCFE$r$ method can be analysed
in the form of continuous-stage RK method \eqref{CRK}, it is indeed
an $r$-stage method in practice. If the integral cannot be directly
calculated, we approximate it by a high-order quadrature rule
$(b_{k},c_{k})_{k=1}^{s}$. The corresponding full discrete scheme of
\eqref{PCRK} is
\begin{equation}\label{DCRK}
\left\{\begin{aligned}
&y_{\sigma}=\sum_{i=1}^{r+1}y_{\frac{i-1}{r}}l_{i}(\sigma),\\
&y_{\frac{i-1}{r}}=y_{0}+h\sum_{k=1}^{s}b_{k}A_{\frac{i-1}{r},c_{k}}f(y_{c_{k}}),\quad i=2,\ldots,r+1.\\
\end{aligned}\right.
\end{equation}
By an argument which is similar to that stated at the beginning of
Section \ref{CRKK}, \eqref{DCRK} is equivalent to a discrete version
of the FFCFE$r$ method \eqref{CFE2}:
\begin{equation*}\label{DCFE}
\left\{\begin{aligned}
&u(0)=y_{0},\\
&\langle v,u^{\prime}\rangle_{\tau}=h[v,f\circ u],\quad u(\tau)\in X_{h},\text{  for all  }v(\tau)\in Y_{h},\\
&y_{1}=u(1),\\
\end{aligned}\right.
\end{equation*}
where $[\cdot,\cdot]$ is the discrete inner product:
$$[w_{1},w_{2}]=[w_{1}(\tau),w_{2}(\tau)]_{\tau}=\sum_{k=1}^{s}b_{k}w_{1}(c_{k})\cdot w_{2}(c_{k}).$$
By the proof procedure of Theorem \ref{symmetry}, one can show that
the full discrete scheme is still symmetric provided the quadrature
rule is symmetric, {i.e. $c_{s+1-k}=1-c_{k}$ and $b_{s+1-k}=b_{k}$
for $k=1,2,\ldots,s$.}

Now it is clear that the practical form \eqref{PCRK} or \eqref{DCRK}
is determined by the Lagrange interpolation functions $l_{i}(\tau)$
and the coefficient $A_{\tau,\sigma}$. For the CFE$r$ method,
\begin{equation*}
Y_{h}=\text{span}\left\{1,\tau,\ldots,\tau^{r-1}\right\},\quad X_{h}=\text{span}\left\{1,\tau,\ldots,\tau^{r}\right\},
\end{equation*}
and all $l_{i}(\tau)$ for $ i=1,2,\ldots,r+1$ are Lagrange
interpolation polynomials of degrees $r$. The $A_{\tau,\sigma}$ for
$r=2,3,4$ are given by
\begin{equation*}
A_{\tau,\sigma}=\left\{\begin{aligned}
&(4+6\sigma(-1+\tau)-3\tau)\tau,\quad r=2,\\
&\tau(9-18\tau+10\tau^{2}+30\sigma^{2}(1-3\tau+2\tau^2)-12\sigma(3-8\tau+5\tau^{2})),\quad r=3,\\
&\tau(16-60\tau+80\tau^{2}-35\tau^{3}+140\sigma^{3}(-1+6\tau-10\tau^{2}+5\tau^{3})\\
&+60\sigma(-2+10\tau-15\tau^{2}+7\tau^{3})-30\sigma^{2}(-8+45\tau-72\tau^{2}+35\tau^{3})),\quad r=4.\\
\end{aligned}\right.
\end{equation*}
For the TFCFE$r$ method,
\begin{equation*}
Y=\text{span}\left\{1,t,\ldots,t^{r-3},\cos(\omega t),\sin(\omega t)\right\},
\end{equation*}
then
\begin{equation*}
Y_{h}=\text{span}\left\{1,\tau,\ldots,\tau^{r-3},\cos(\nu\tau),\sin(\nu\tau)\right\},\quad
X_{h}=\text{span}\left\{1,\tau,\ldots,\tau^{r-2},\cos(\nu\tau),\sin(\nu\tau)\right\},
\end{equation*}
where $\nu=h\omega$. The corresponding $A_{\tau,\sigma}$ and
$l_{i}(\tau)$ are more complicated than those of CFE$r$, but one can
calculate them by the formulae \eqref{Aexplicit} and
\eqref{Lagrange} without any difficulty before solving the IVP
numerically. Thus the computational cost of the TFCFE$r$ method at
each step is comparable to that of the CFE$r$ method. Besides, when
$\nu$ is small, to avoid unacceptable cancellation, it is
recommended to calculate variable coefficients in TF methods by
their Taylor expansions with respect to $\nu$ at zero.

\section{Numerical experiments}\label{NE}
In this section, we carry out four numerical experiments to test the
effectiveness and efficiency of the new methods TFCFE$r$ based on
the space \eqref{TF1CFE} for $r=2,3,4$ {and TF2CFE4 based on the
space \eqref{TF2CFE}} in the long-term computation of structure
preservation. These new methods are compared with standard $r$-stage
$2r$th-order EP {CFE$r$ methods} for $r=2,3,4$. Other methods such
as the $2$-stage $4$th-order EF symplectic Gauss-Legendre
collocation method (denoted by EFGL$2$) derived in \cite{Calvo2008}
and the $2$-stage $4$th-order EF EP method (denoted by EFCRK$2$)
derived in \cite{Miyatake2014} are also considered. Since all of
these structure-preserving methods are implicit, fixed-point
iterations are needed to solve the nonlinear algebraic systems at
each step. The tolerance error for the iteration solution is set to
$10^{-15}$ in the numerical simulation.

Numerical quantities with which we are mainly concerned are the Hamiltonian error
$$EH=(EH^{0},EH^{1},\ldots),$$
with
$$EH^{n}=|H(y_{n})-H(y_{0})|,$$
and the solution error
$$ME=(ME^{0},ME^{1},\ldots),$$
with
$$ME^{n}=||y_{n}-y(t_{n})||_{\infty}.$$
Correspondingly, the maximum global errors of Hamiltonian (GEH) and
the solution (GE) are defined by:
$$GEH=\max_{n\geq0}EH^{n},\quad GE=\max_{n\geq0}ME^{n},$$
respectively. Here the numerical solution at the time node $t_{n}$
is denoted by $y_{n}$.

\begin{myexp}\label{exp1}
Consider the Perturbed Kepler problem defined by the Hamiltonian:
\begin{equation*}
H(p,q)=\frac{1}{2}(p_{1}^{2}+p_{2}^{2})-\frac{1}{(q_{1}^{2}+q_{2}^{2})^{\frac{1}{2}}}-\frac{2\varepsilon+\varepsilon^{2}}{3(q_{1}^{2}+q_{2}^{2})^{\frac{3}{2}}},
\end{equation*}
with the initial condition $q_{1}(0)=1,q_{2}=0,p_{1}(0)=0,p_{2}=1+\varepsilon,$  where $\varepsilon$ is a small parameter.
The exact solution of this IVP is
$$q_{1}(t)=\cos((1+\varepsilon)t),\quad q_{2}(t)=\sin((1+\varepsilon)t),\quad p_{i}(t)=q_{i}^{\prime}(t),\ i=1,2.$$
Taking $\omega=1$, $\varepsilon=0.001$ and $h=1/2^{i},i=-1,0,\ldots,6,$ we integrate this problem over the
interval $[0,200\pi]$ by the TF2CFE4, TFCFE$r$ and CFE$r$ methods for $r=2,3,4.$
The nonlinear integral in the $r$-stage method is evaluated by the
$(r+1)$-point Gauss-Legendre quadrature rule. Numerical results are presented
in Fig. \ref{PKEPLER}.

From Fig. \ref{PKEPLER}(a) it can be observed
that TFCFE$r$ and TF2CFE4 methods show the expected order.
Under the same stepsize, the TF method is more accurate than the
non-TF method of the same order. Since the double precision provides
only $16$ significant digits, the numerical solution are
polluted significantly by rounding errors when the maximum global
error attains the magnitude $10^{-11}$. Fig. \ref{PKEPLER}(b) shows
that the efficiency of the TF method is higher than that of the
non-TF method of the same order. Besides, high-order methods are
more efficient than low-order ones when the stepsize is relatively
small.

In Fig. \ref{PKEPLER}(c), one can see that all of these EP methods preserve the Hamiltonian very well. The error in the Hamiltonian
are mainly contributed by the quadrature error when the stepsize $h$ is large and the rounding error when $h$ is small.

\begin{figure}[ptb]
\centering

  \begin{tabular}[c]{cccc}%
  \subfigure[]{\includegraphics[width=5cm,height=7cm]{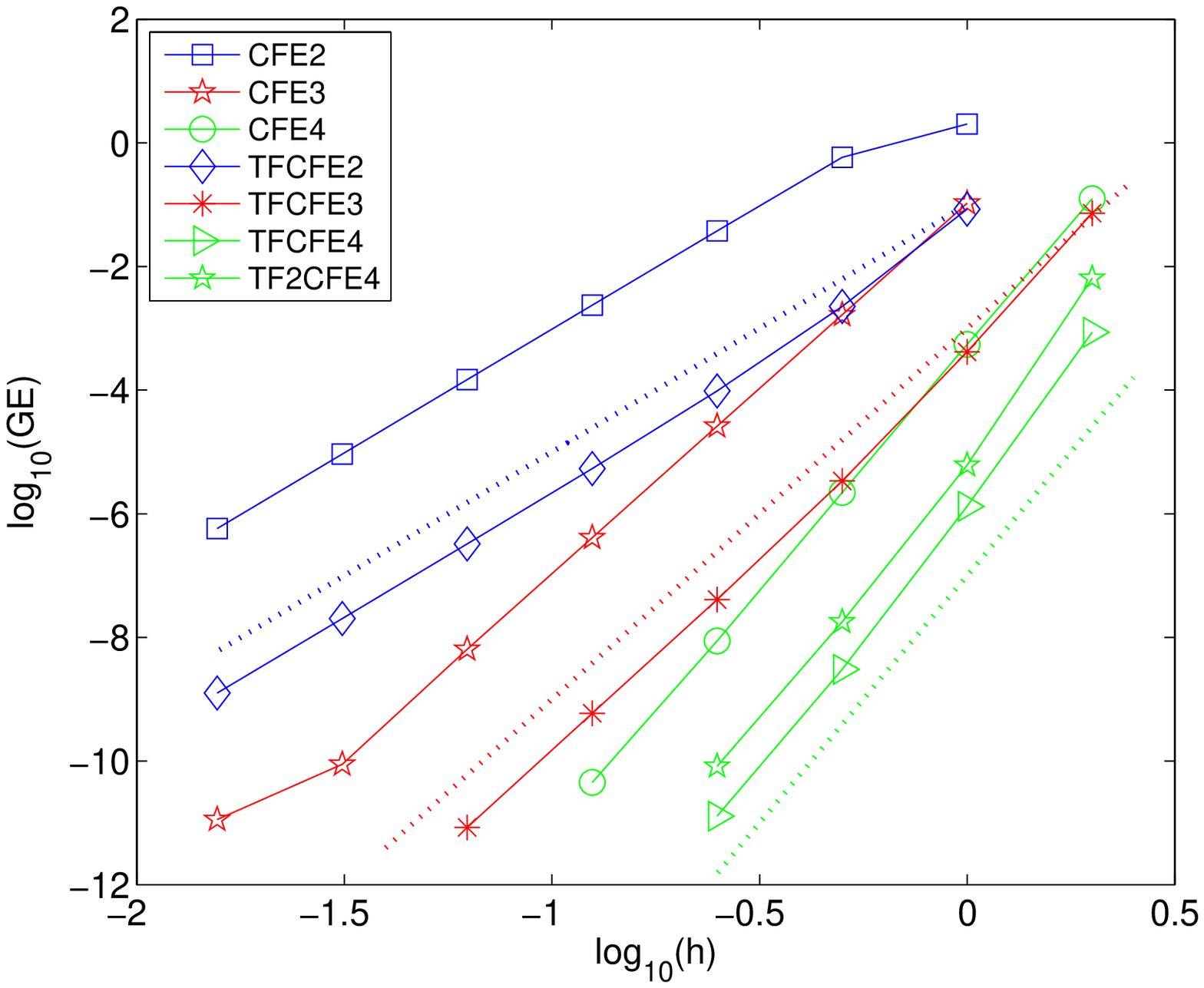}}
  \subfigure[]{\includegraphics[width=5cm,height=7cm]{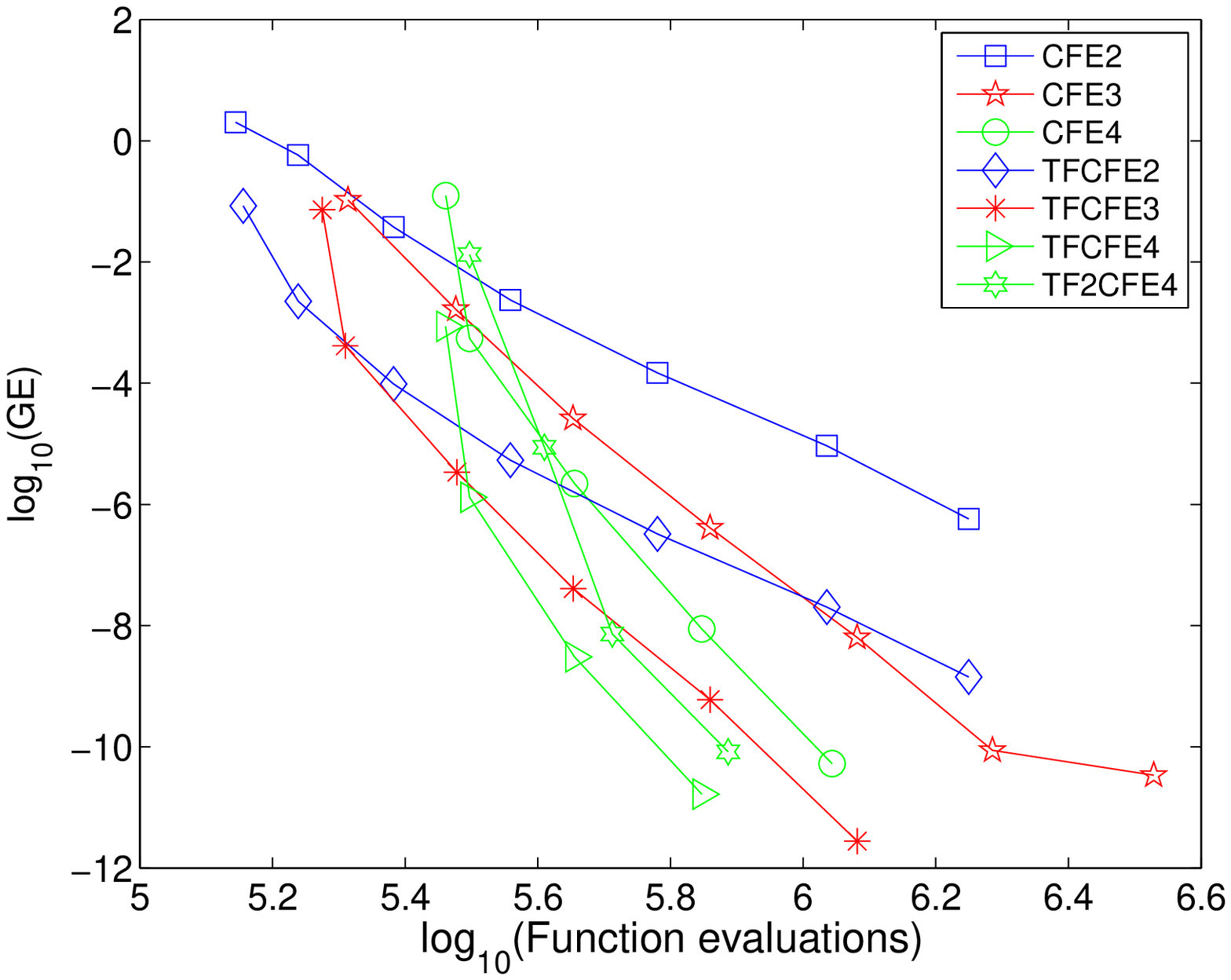}}
  \subfigure[]{\includegraphics[width=5cm,height=7cm]{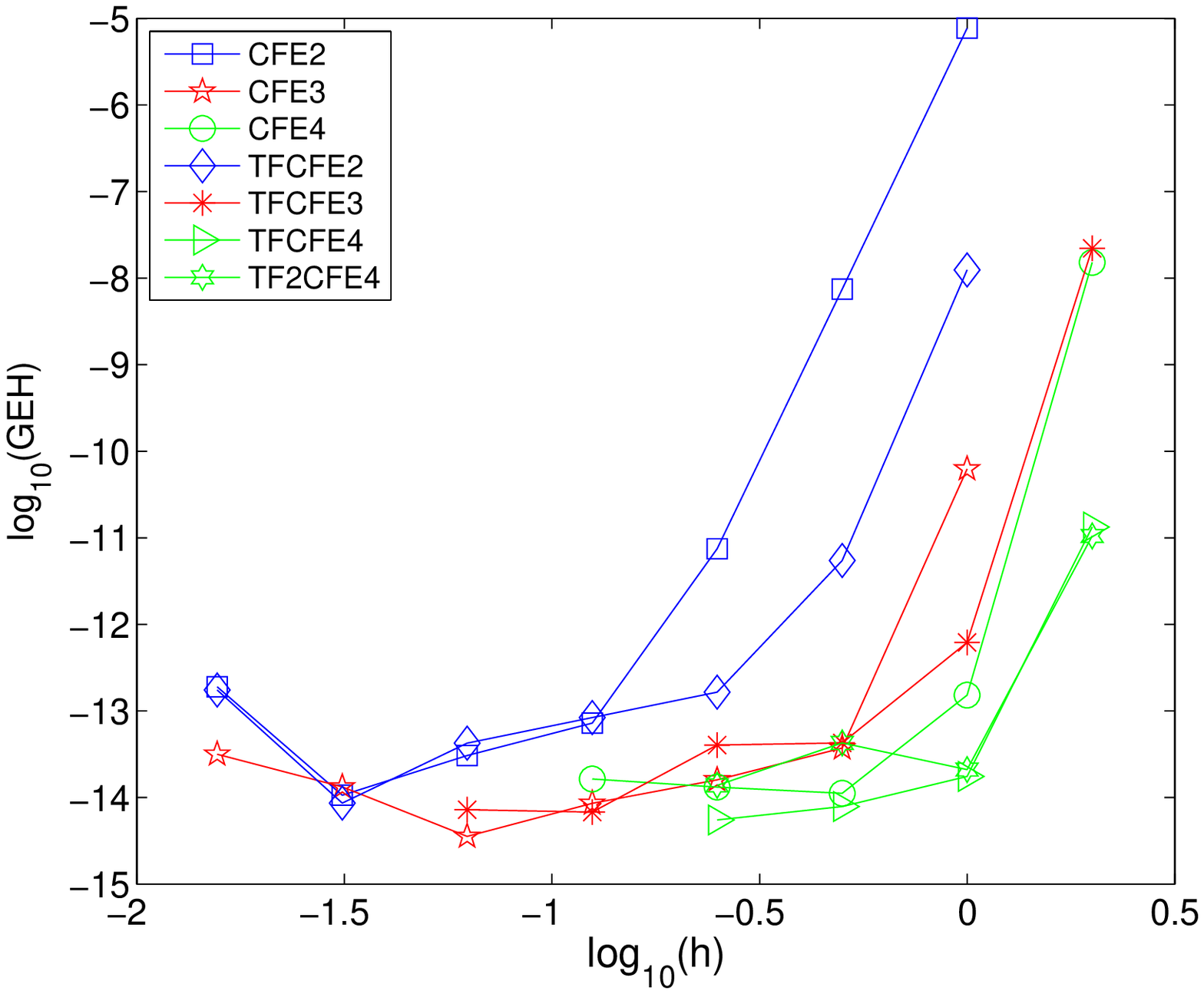}}
  \end{tabular}

\caption{(a) The logarithm of the maximum global error against the logarithm of the stepsize. The dashed lines have slopes four, six and eight.
(b) The logarithm of the maximum global error against the logarithm of function evaluations.
(c) The logarithm of the maximum global error of Hamiltonian against the logarithm of the stepsize. }
\label{PKEPLER}
\end{figure}
\end{myexp}

\begin{myexp}\label{exp2}
Consider the Duffing equation defined by the Hamiltonian :
\begin{equation*}
H(p,q)=\frac{1}{2}p^{2}+\frac{1}{2}(\omega^{2}+k^{2})q^{2}-\frac{k^{2}}{2}q^{4}
\end{equation*}
with the initial value $q(0)=0, p(0)=\omega$. The exact solution of this IVP is
$$q(t)=sn(\omega t;k/\omega),\quad p(t)=cn(\omega t;k/\omega)dn(\omega t;k/\omega).$$
{where $sn, cn$ and $dn$ are Jacobi elliptic functions.} Taking
$k=0.07, \omega=5$ and $h=1/5\times1/2^{i},i=0,1,\ldots,5,$ we
integrate this problem over the interval $[0,100]$ by TFCFE$2$,
TFCFE$3$, CFE$2$, CFE$3$ and EFCRK$2$ methods. Since the nonlinear term $f$
is polynomial, we can calculate the integrals
involved in these methods exactly by Mathematica at the beginning of the computation. Numerical results
are shown in Figs. \ref{DUFFING}.

In Fig. \ref{DUFFING}(a), one can see that the TF method is more
accurate than the non-TF method of the same order under the same
stepsize. Both as $2$-stage $4th$-order methods, TFCFE$2$ method is more
accurate than EFCRK$2$ method in this problem. Again, it can be
observed from Fig. \ref{DUFFING}(b) that the efficiency of the CFE$r$
method is lower than that of the EF/TF method of the same order.
Although the nonlinear integrals are exactly calculated in theory,
Fig. \ref{DUFFING}(c) shows that all of these method only approximately
preserve the Hamiltonian. It seems that the rounding
error increases as $h\to0$.

\begin{figure}[ptb]
\centering

  \begin{tabular}[c]{cccc}%
  \subfigure[]{\includegraphics[width=4.5cm,height=7cm]{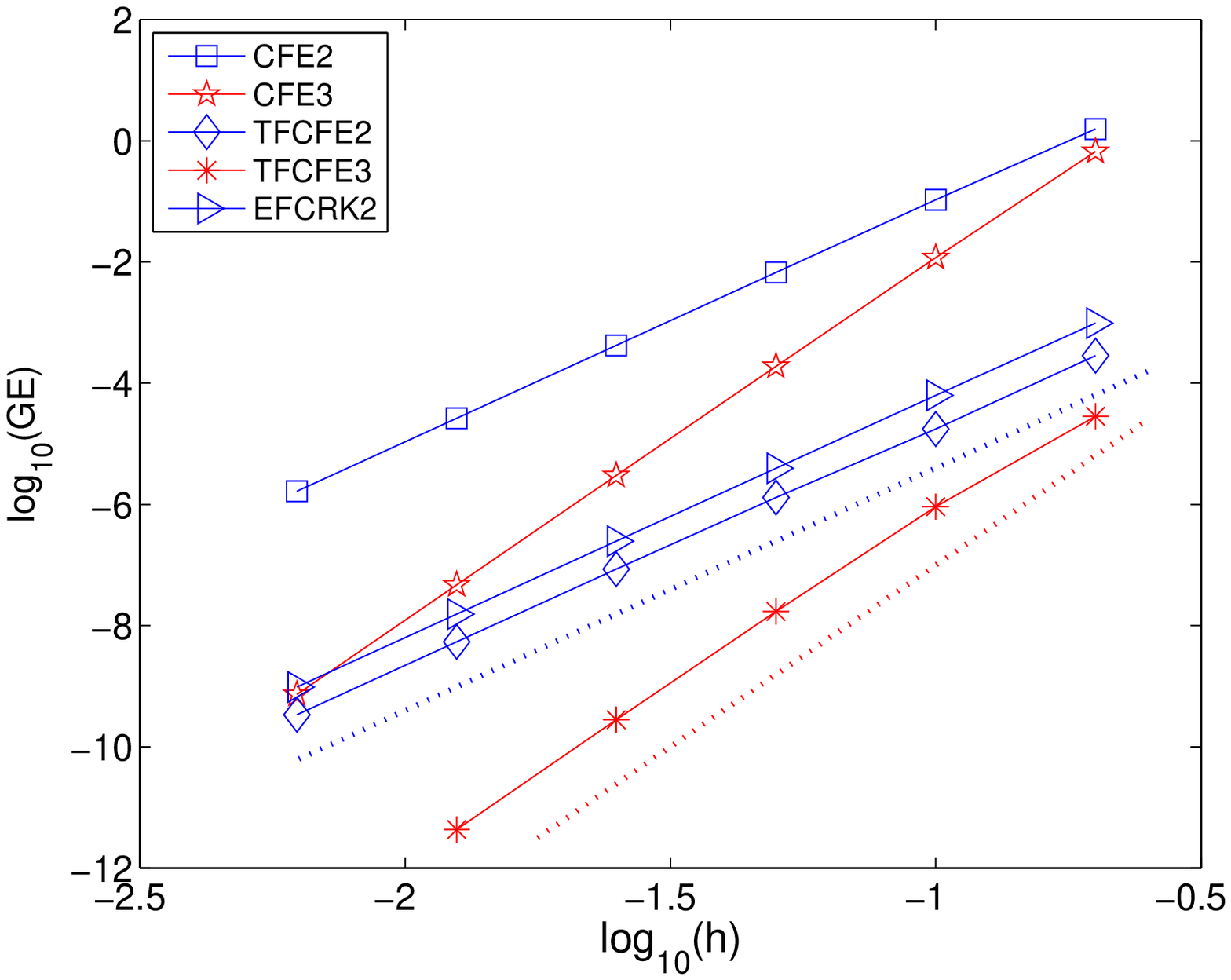}}
  \subfigure[]{\includegraphics[width=4.5cm,height=7cm]{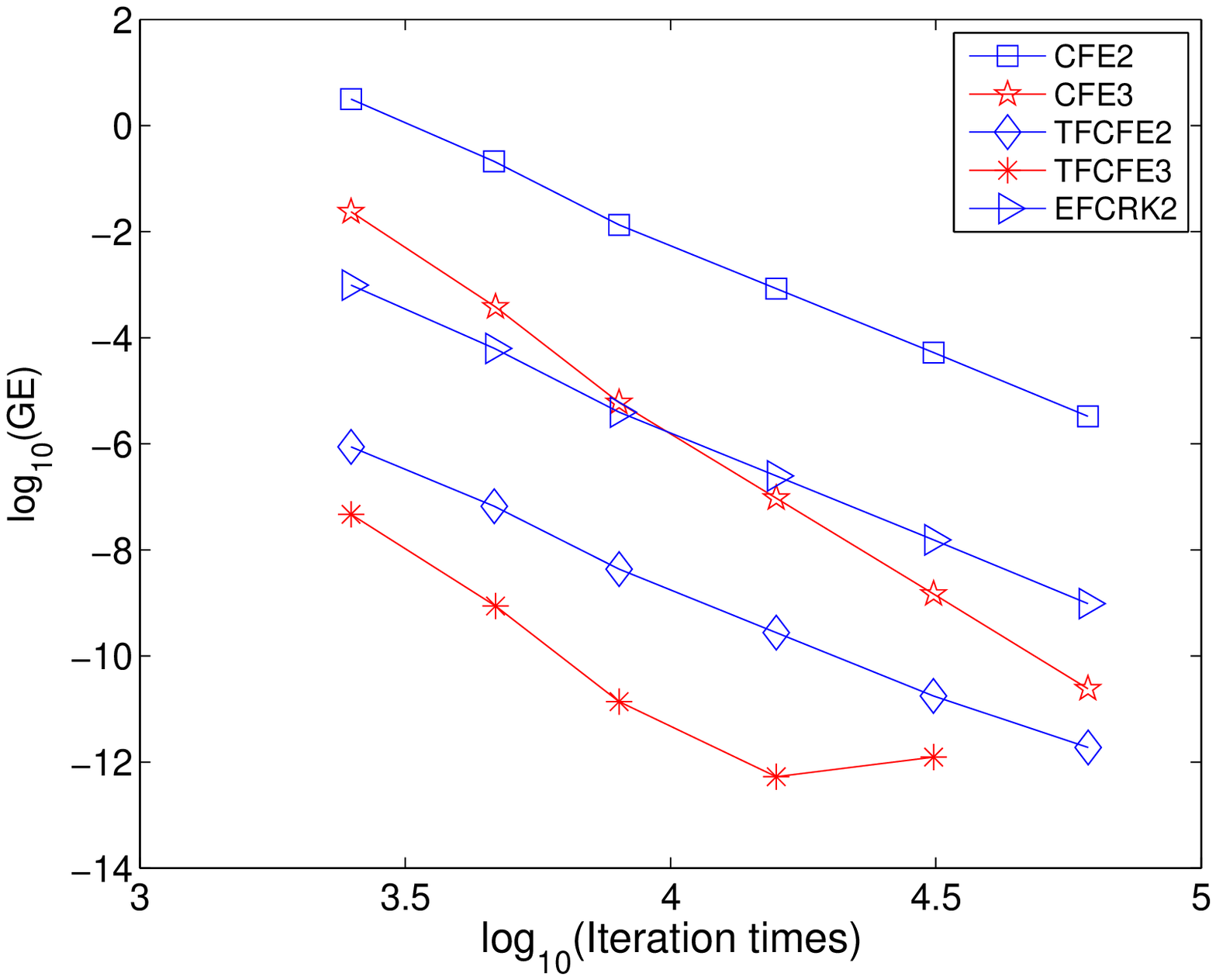}}
  \subfigure[]{\includegraphics[width=4.5cm,height=7cm]{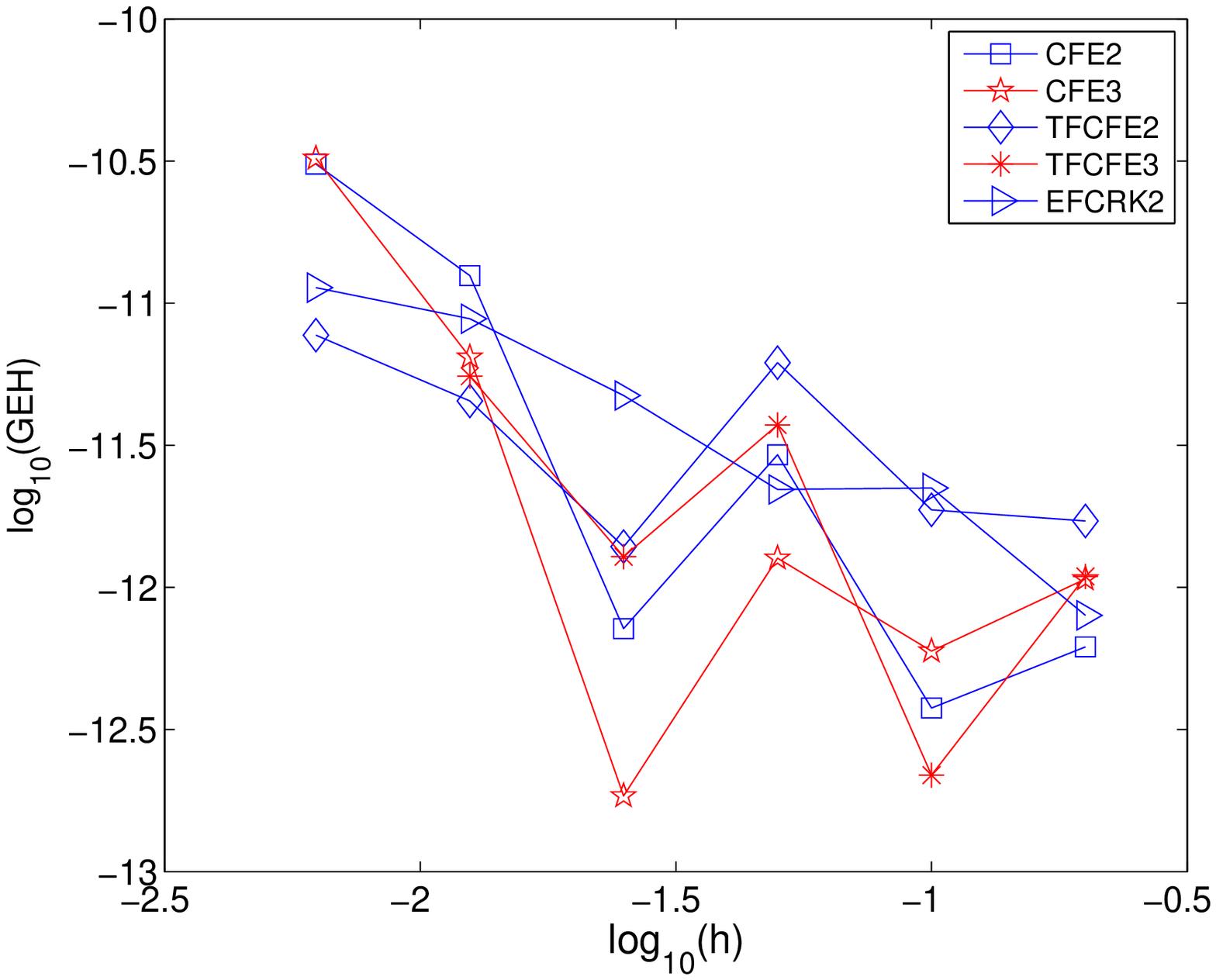}}
  \end{tabular}

\caption{(a) The logarithm of the maximum global error against the logarithm of the stepsize. The dashed lines
have slopes four and six. (b) The logarithm of the maximum global error against the logarithm of iteration times.
(c) The logarithm of the maximum global error of Hamiltonian against the logarithm of the stepsize. }
\label{DUFFING}
\end{figure}

\begin{myexp}
Consider the Fermi-Pasta-Ulam problem studied by Hairer, {\em et al} in
\cite{Hairer2000,Hairer2006}, which is defined by the Hamiltonian
\begin{equation*}
H(p,q)=\frac{1}{2}p^{\intercal}p+\frac{1}{2}q^{\intercal}Mq+U(q),
\end{equation*}
where

\begin{gather*}
M=\left(
\begin{array}
[c]{cc}%
O_{m\times m} & O_{m\times m}\\
O_{m\times m} & \omega^{2}I_{m\times m}%
\end{array}
\right)  ,\\
U(q)=\dfrac{1}{4}\left((q_{1}-q_{m+1})^{4}+\textstyle\sum\limits_{i=1}%
^{m-1}(q_{i+1}-q_{m+i+1}-q_{i}-q_{m+i})^{4}+(q_{m}+q_{2m})^{4}\right).
\end{gather*}
{In this problem, we choose $m=2, q_{1}(0)=1, p_{1}(0)=1,
q_{3}(0)=1/\omega, p_{3}(0)=1,$ and zero for the remaining initial
values.} Setting $\omega=50, h=1/50$ and $\omega=100, h=1/100$, we integrate this
problem over the interval $[0,100]$ by CFE2, CFE3, TFCFE2, TFCFE3 and EFCRK2 methods.
The nonlinear integrals are calculated exactly by Mathematica at the beginning of
the computation. We choose the numerical solution obtained by a high-order
method with a sufficiently small stepsize as the `reference
solution' in the FPU problem. Numerical results are plotted in Fig.
\ref{FPU}.

In Figs. \ref{FPU}(a), \ref{FPU}(c), one can see that the TF
methods are more accurate than non-TF ones. Unlike the previous
problem, the EFCRK$2$ method wins slightly over TFCFE$2$ method in this
case. And Figs. \ref{FPU}(b), \ref{FPU}(d) show that all of these
methods display promising EP property ,
which is especially important in the FPU problem.
\end{myexp}
\begin{figure}[ptb]
\centering
\begin{tabular}[c]{cccc}%
  \subfigure[]{\includegraphics[width=5.5cm,height=7cm]{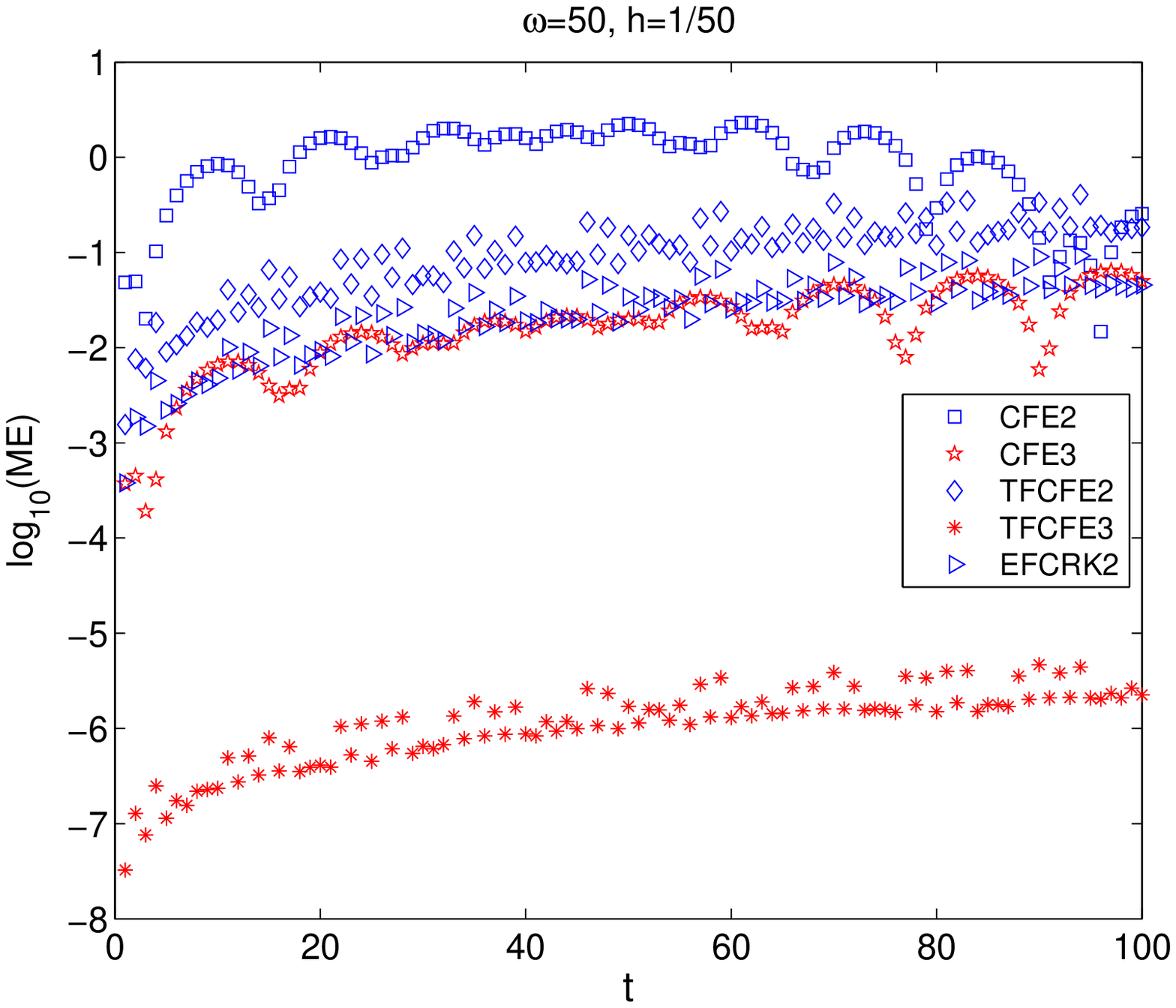}}
  \subfigure[]{\includegraphics[width=5.5cm,height=7cm]{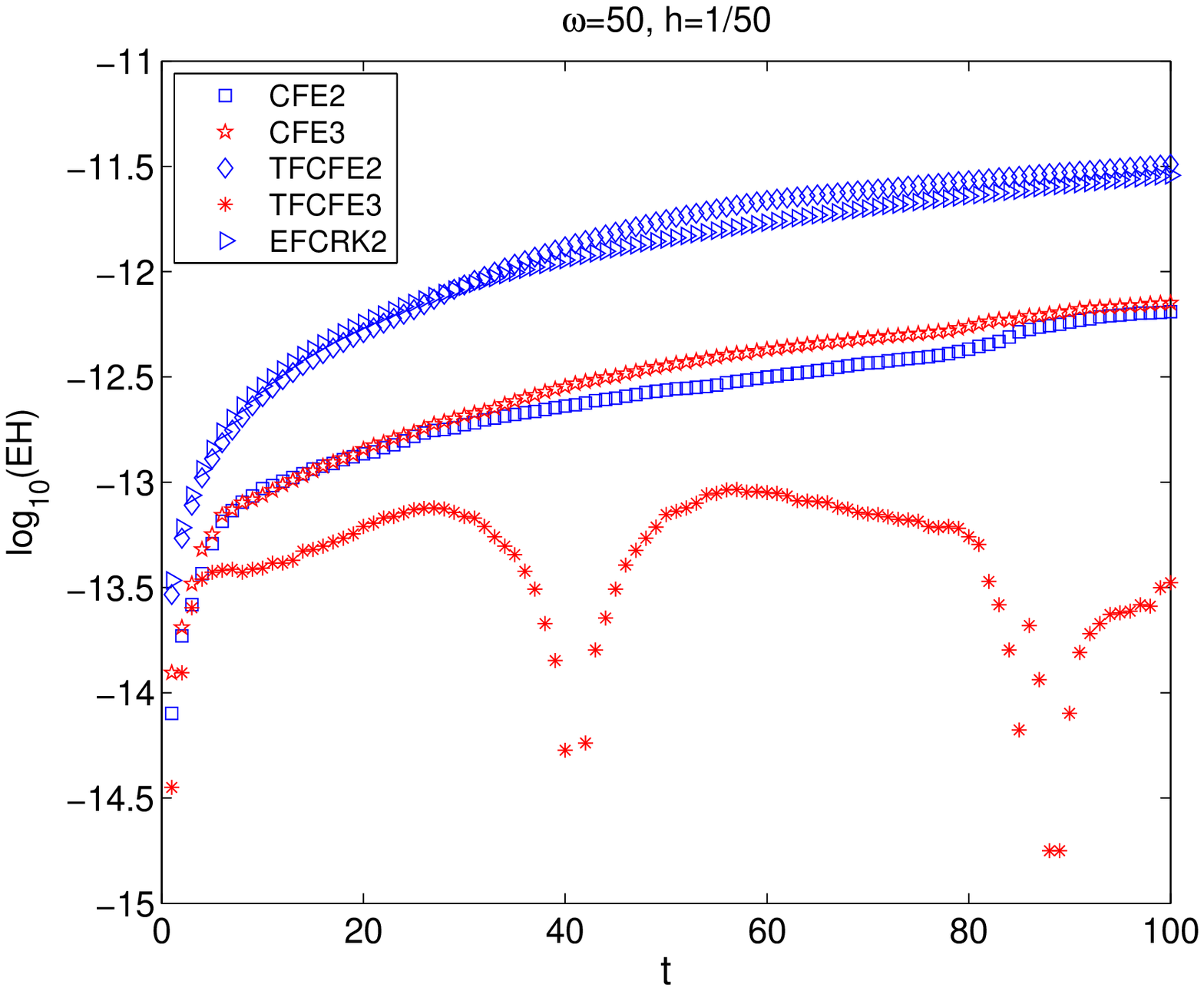}}
\end{tabular}
\begin{tabular}[c]{cccc}%
  \subfigure[]{\includegraphics[width=5.5cm,height=7cm]{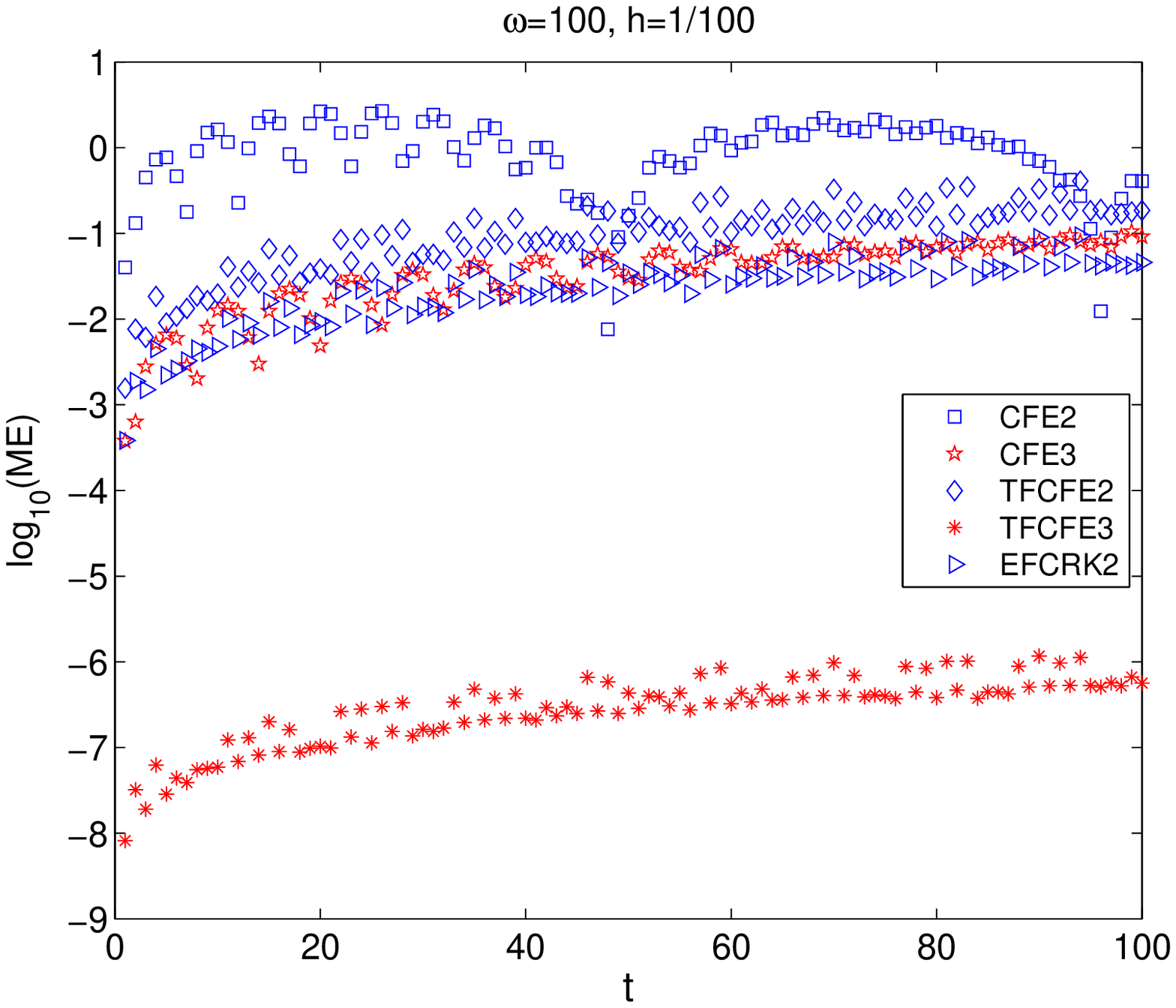}}
  \subfigure[]{\includegraphics[width=5.5cm,height=7cm]{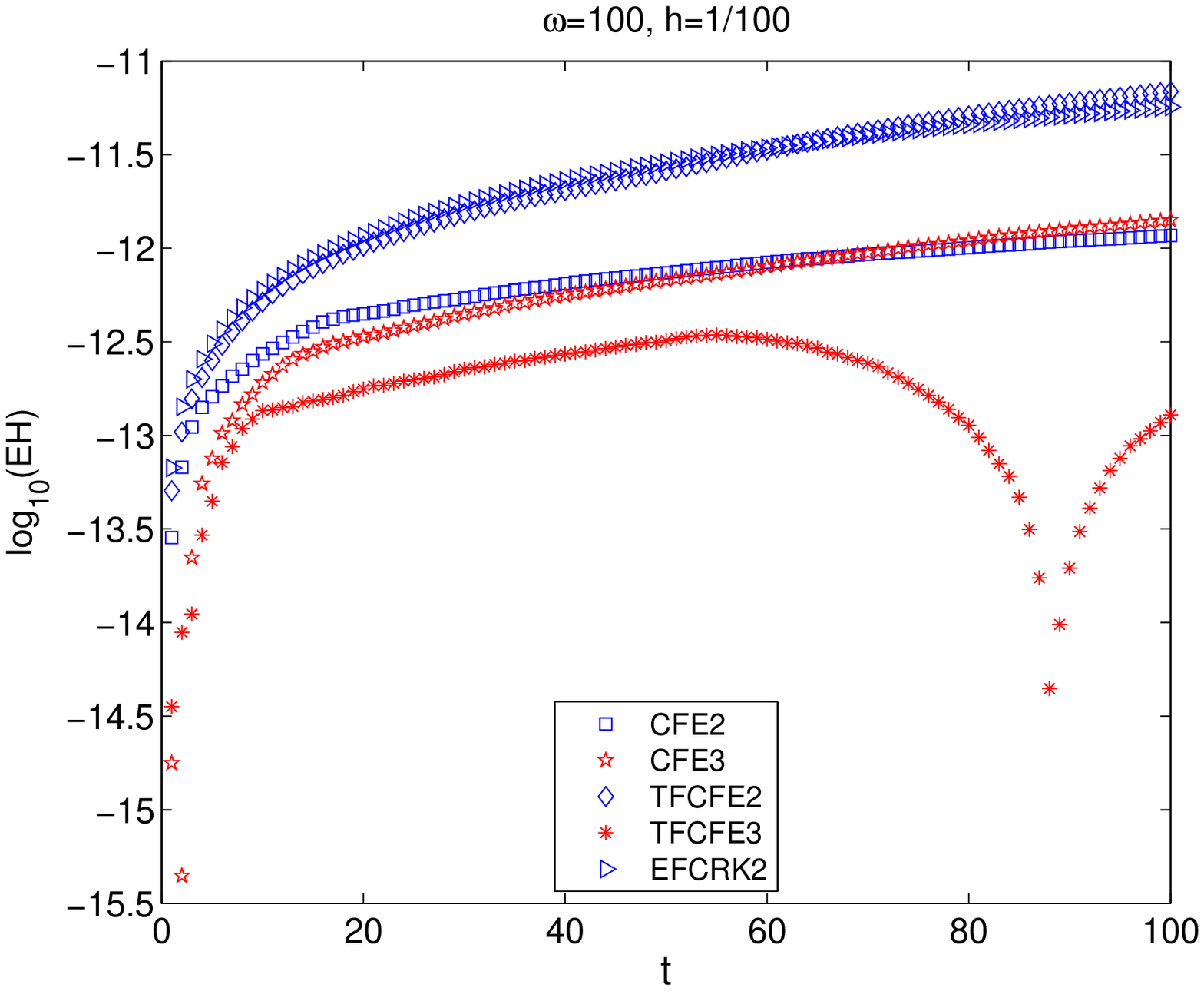}}
\end{tabular}
\caption{(a) (c) The logarithm of the solution error against
time $t$. (b) (d) The logarithm of the Hamiltonian error
against time $t$.} \label{FPU}
\end{figure}

\end{myexp}

\begin{myexp}\label{exp3}
Consider the IVP defined by the nonlinear Schr\"{o}dinger equation
\begin{equation}\label{NLSE}
\left\{\begin{aligned}
&iu_{t}+u_{xx}+2|u|^{2}u=0,\\
&u(x,0)=\varphi(x),\\
\end{aligned}\right.
\end{equation}
where $u$ is a complex function of $x, t$, and $i$ is the imaginary
unit. {Taking the periodic boundary condition
$u(x_{0},t)=u(x_{0}+L,t)$ and discretizing the spatial derivative
$\partial_{xx}$ by the pseudospectral method (see e.g.
\cite{Chen2001}),} this problem is converted into a complex system
of ODEs :
\begin{equation*}
\left\{\begin{aligned}
&i\frac{d}{dt}U+D^{2}U+2|U|^{2}\cdot U=0,\\
&U(0)=(\varphi(x_{0}),\varphi(x_{1}),\ldots,\varphi(x_{d-1}))^{\intercal},\\
\end{aligned}\right.
\end{equation*}
or an equivalent Hamiltonian system:
\begin{equation}\label{semi}
\left\{\begin{aligned}
&\frac{d}{dt}P=-D^{2}Q-2(P^{2}+Q^{2})\cdot Q,\\
&\frac{d}{dt}Q=D^{2}P+2(P^{2}+Q^{2})\cdot P,\\
&P(0)=real(U(0)),\quad Q(0)=imag(U(0)),\\
\end{aligned}\right.
\end{equation}
where the superscript `2' is the entrywise square multiplication
operation for vectors, $x_{n}=x_{0}+n\Delta x/d,\ n=0,1,\ldots,d-1,
U=(U_{0}(t),U_{1}(t),\ldots,U_{d-1}(t))^{\intercal},
P(t)=real(U(t)), Q(t)=imag(U(t))$ and $D=(D_{jk})_{0\leq j,k\leq
d-1}$ is the pseudospectral differential matrix defined by:
\begin{equation*}
D_{jk}=\left\{\begin{aligned}
&\frac{\pi}{L}(-1)^{j+k}cot(\pi\frac{x_{j}-x_{k}}{L}),\ j\neq k,\\
&0,\quad\quad\quad\quad\quad\quad\quad\quad\quad j=k.\\
\end{aligned}\right.
\end{equation*}
The Hamiltonian or the total energy of \eqref{semi} is
\begin{equation*}
H(P,Q)=\frac{1}{2}P^{\intercal}D^{2}P+\frac{1}{2}Q^{\intercal}D^{2}Q+{\frac{1}{2}\sum_{i=0}^{d-1}(P_{i}^{2}+Q_{i}^{2})^{2}.}
\end{equation*}
In \cite{Peregrine}, the author constructed a periodic bi-soliton
solution of \eqref{NLSE}:
\begin{equation}\label{exact}
u(x,t)=\frac{U(x,t)}{V(x,t)},
\end{equation}
where
\begin{equation*}
\begin{aligned}
&U(x,t)=\exp(iM^{2}t)M\cosh^{-1}(M(x-A))-\exp(iN^{2}t)N\cosh^{-1}(N(x+A)),\\
&V(x,t)=\cosh(J)-\sinh(J)(\tanh(M(x-A))\tanh(N(x+A))\\
&+\cos((M^{2}-N^{2})t)\cosh^{-1}(M(x-A))\cosh^{-1}(N(x+A))),\\
\end{aligned}
\end{equation*}
with
\begin{equation*}
J=\tanh^{-1}(\frac{2MN}{M^{2}+N^{2}}).
\end{equation*}
This solution can be viewed approximately as the superposition of two single solitons located at $x=A$ and $x=-A$ respectively. Since it decays exponentially when $x\to\infty$, we can take the periodic boundary condition for sufficiently small $x_{0}$ and large $L$ with little loss of accuracy. Aside from the total energy, it is well known that the semi-discrete NLS \eqref{semi} has another invariant, the total charge
$$C(P,Q)=\sum_{i=0}^{d-1}(P_{i}^{2}+Q_{i}^{2}).$$
Thus we also calculate the error in the charge(EC):
$$EC=(EC^{0},EC^{1},\ldots)$$
with
$$EC^{n}=|C^{n}-C^{0}|,\quad C^{n}=C(P^{n},Q^{n})$$
{where $P^{n}\approx P(t_{n}), Q^{n}\approx Q(t_{n})$ is the
numerical solution at the time node $t_{n}$.} Taking $x_{0}=-50,
L=100, A=10, M=1, N=2^{\frac{1}{2}}, d=450, h=0.2, \omega=2, $ we
integrate the semi-discrete problem \eqref{semi} by TFCFE$2$, CFE$2$ and
EFGL$2$ methods over the time interval $[0,100]$. The nonlinear
integrals are calculated exactly by Mathematica at the beginning of
the computation. Numerical results are presented in Fig. \ref{NLS}.

It is noted that the exact solution \eqref{exact} has two
approximate frequency $M^{2}$ and $N^{2}$. By choosing the larger
frequency $N^{2}=2$ as the fitting frequency $\omega$, the EF/TF methods
still reach higher accuracy than the general-purpose method CFE2, see Fig.
\ref{NLS}(a). Among three EF/TF methods, TFCFE2 is the most accurate.
Fig. \ref{NLS}(b) shows that three EP methods CFE2, TFCFE2
and EFCRK2 preserve the Hamiltonian (apart from the rounding error). Since
EFGL2 is a symplectic method, it preserves
the discrete charge, which is a quadratic invariant, see Fig. \ref{NLS}(c).
Although TFCFE2 method cannot preserve the discrete
charge, its error in the charge is smaller than the charge error of
CFE2 and EFCRK2.

\begin{figure}[ptb]
\centering
\begin{tabular}[c]{cccc}%
  \subfigure[]{\includegraphics[width=4.5cm,height=7cm]{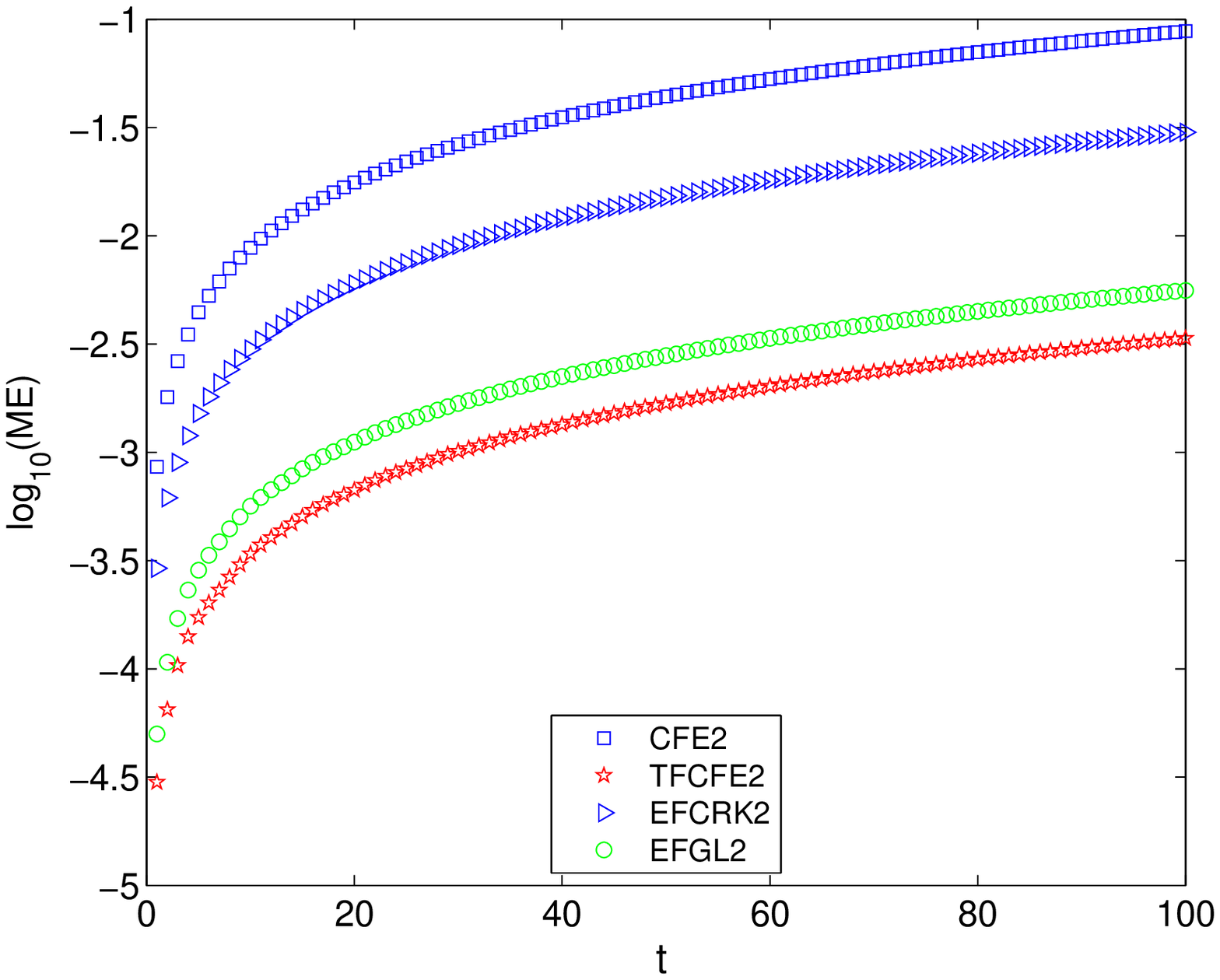}}
  \subfigure[]{\includegraphics[width=4.5cm,height=7cm]{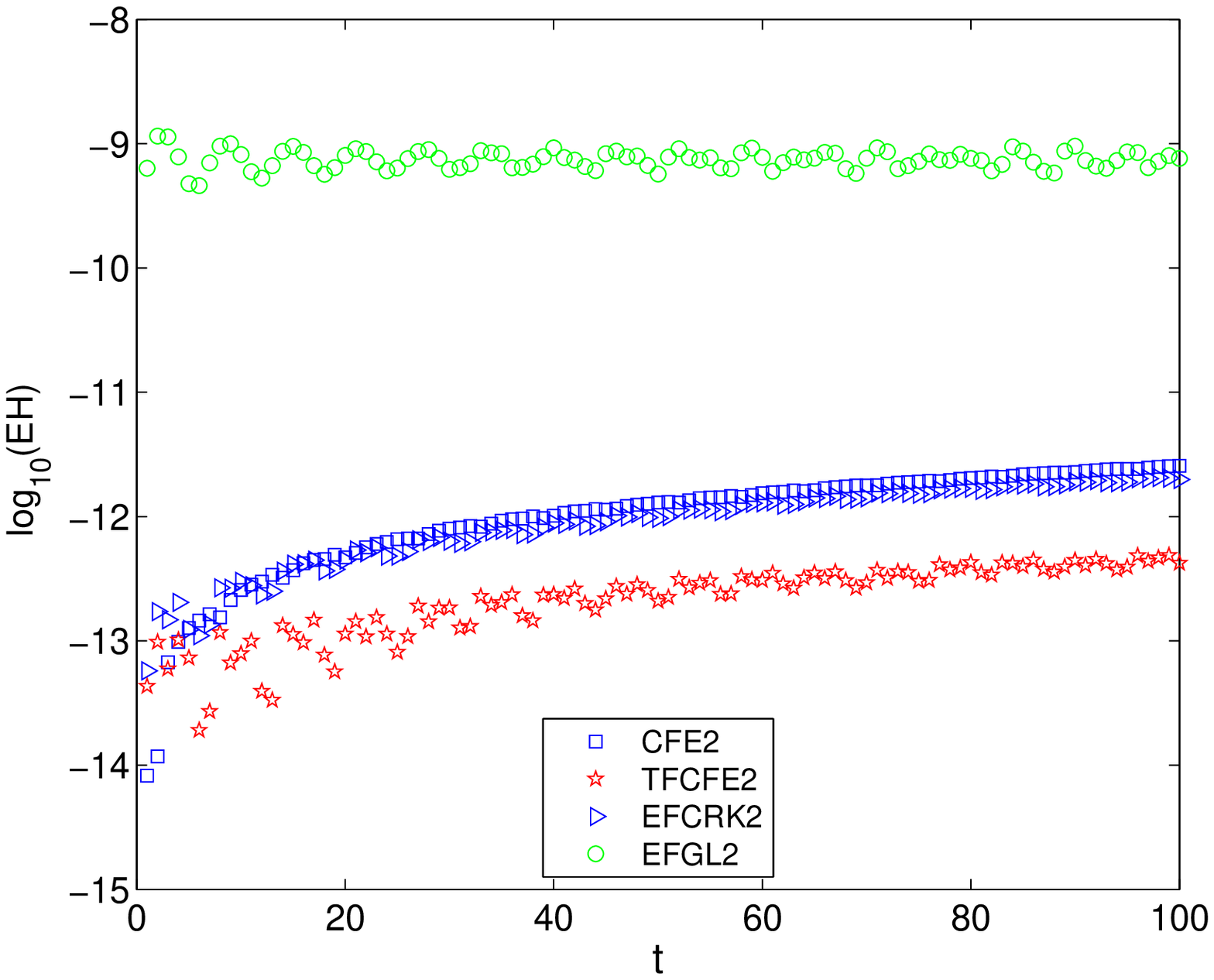}}
  \subfigure[]{\includegraphics[width=4.5cm,height=7cm]{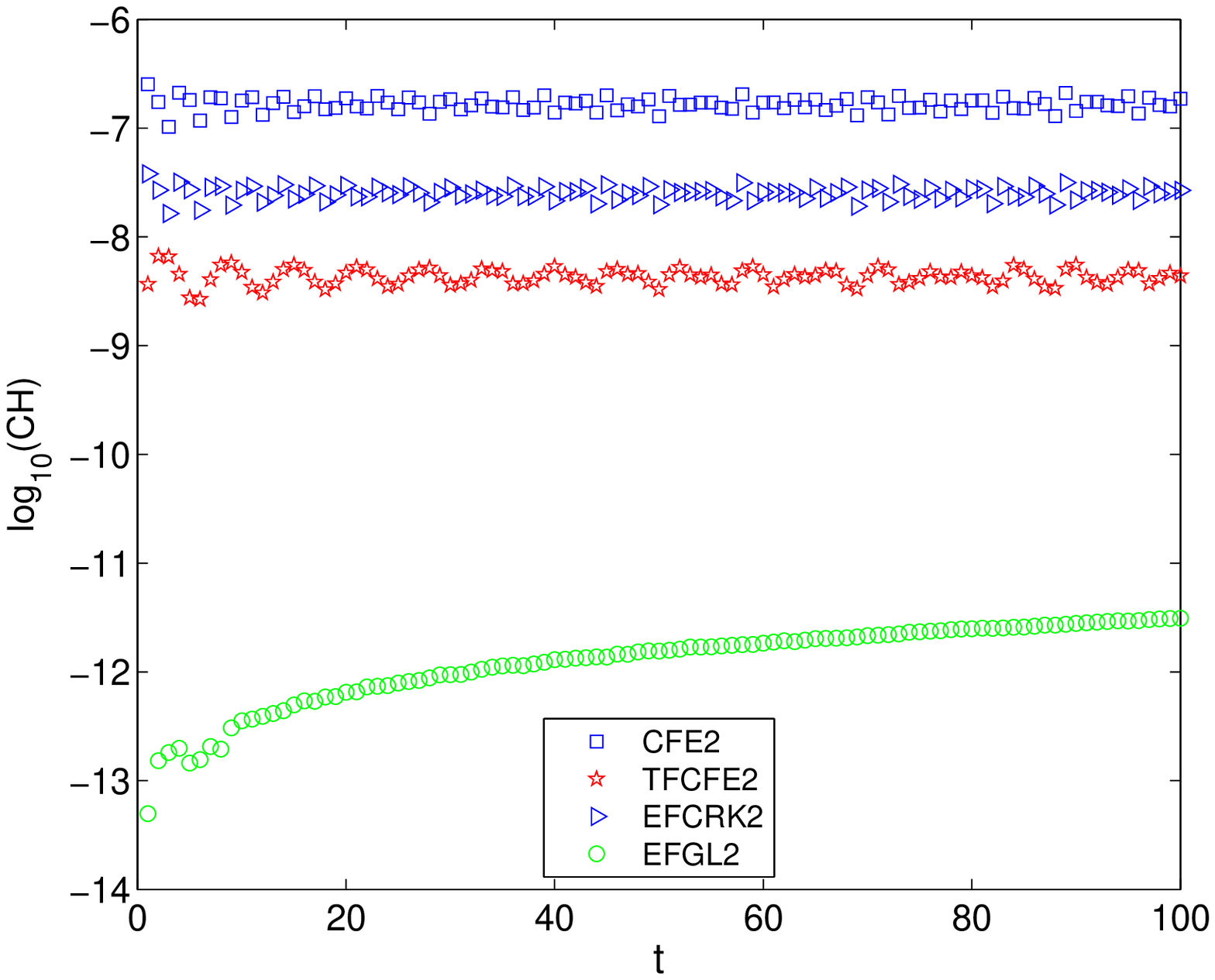}}
\end{tabular}
\caption{(a) The logarithm of the solution error against time $t$. (b) The logarithm of the Hamiltonian error against time $t$.
(c) The logarithm of the charge error against time $t$.}
\label{NLS}
\end{figure}
\end{myexp}

\section{Conclusions and discussions}

Oscillatory systems constitute an important category of differential
equations in numerical simulations. {It should be noted that} the
numerical treatment of oscillatory systems is full of challenges.
This paper is mainly concerned with the establishment of high-order
functionally-fitted energy-preserving methods for solving
oscillatory nonlinear Hamiltonian systems. To this end, we have
derived new FF EP methods FFCFE$r$ based on the analysis of
continuous finite element methods. The FFCFE$r$ method can be
thought of as the continuous-stage Runge--Kutta methods, therefore
they can be used conveniently in applications. The geometric
properties and algebraic orders of them have been analysed in
detail. By equipping FFCFE$r$ with the spaces \eqref{TF1CFE} and
\eqref{TF2CFE}, we have developed the TF EP methods denoted by
TFCFE$r$ and TF2CFE$r$ which are suitable for solving oscillatory
Hamiltonian systems with a fixed frequency $\omega$. Evaluating the
nonlinear integrals in the EP methods exactly or approximately, we
have compared TFCFE$r$ for $r=2,3,4$ and TF2CFE4 with other
structure-preserving methods such as EP methods CFE$r$ for
$r=2,3,4$, the EP method EFCRK2 and the symplectic method EFGL2. It
can be observed from the numerical results that the newly derived TF
EP methods show definitely a high accuracy, an excellent
invariant-preserving property and a prominent long-term behaviour.

In this paper, our numerical experiments are  mainly concerned with
the TFCFE$r$ method and oscillatory Hamiltonian systems. However,
the FFCFE$r$ method is still symmetric and of order $2r$ for the
general autonomous system $y^{\prime}(t)=f(y(t))$. By choosing
appropriate function spaces, the FFCFE$r$ method can be applied {to
solve a much wider class of dynamic systems in applications.} For
example, the numerical experiment concerning the application of FF
Runge--Kutta method to the stiff system has been shown in
\cite{Ozawa2001}. Consequently, the FFCFE$r$ method is likely to be
a highly flexible method with broad prospects.

\section*{Acknowledgments.}

The authors are sincerely indebted to two anonymous referees for
their valuable suggestions, which help improve the presentation of
the manuscript.

\end{document}